\documentclass[a4paper,12pt]{article}
\makeatletter
\newcommand*{\rom}[1]{\expandafter\@slowromancap\romannumeral #1@}
\makeatother

\usepackage{lipsum,relsize,graphicx,amsmath,amsthm,amssymb,mathtools,stmaryrd }
\usepackage[margin=1in]{geometry}
\usepackage{hyperref}
\usepackage{tikz-cd}

\usepackage[toc,page]{appendix}
\usepackage{tocloft}
\usepackage{setspace}

\DeclareSymbolFont{largesymbols}{OMX}{yhex}{m}{n}
\DeclareMathAccent{\wideparen}{\mathord}{largesymbols}{"F3}

\newcounter{tmp}

\DeclareMathOperator{\ad}{ad}
\DeclareMathOperator{\im}{im}

\newtheorem{theorem}{Theorem}[section]
\newtheorem{lemma}[theorem]{Lemma}
\newtheorem{corollary}[theorem]{Corollary}
\newtheorem{definition}[theorem]{Definition}
\newtheorem{proposition}[theorem]{Proposition}
\newtheorem{conjecture}[theorem]{Conjecture}
\newtheorem{construction}[theorem]{Construction}
\newtheorem{theo}{Theorem}[section]

\begin{document}

\title{Primitive ideals in rational, nilpotent Iwasawa algebras}

\author{Adam Jones}

\date{\today}

\maketitle

\begin{abstract}

\noindent Given a $p$-adic field $K$ and a nilpotent uniform pro-$p$ group $G$, we prove that  all primitive ideals in the $K$-rational Iwasawa algebra $KG$ are maximal, and can be reduced to a particular standard form. Setting $\mathcal{L}$ as the associated $\mathbb{Z}_p$-Lie algebra of $G$, our approach is to study the action of $KG$ on a Dixmier module $\widehat{D(\lambda)}$ over the affinoid envelope $\widehat{U(\mathcal{L})}_K$, and to prove that all primitive ideals can be reduced to annihilators of modules of this form.

\end{abstract}

\tableofcontents

\section{Introduction}

Fix $p>2$ a prime, and let $K\backslash\mathbb{Q}_p$ be a finite extension with ring of integers $\mathcal{O}$, uniformiser $\pi$, residue field $k$. 

\subsection{Background}

Let $G$ be a compact $p$-adic Lie group, and recall that we define the \emph{completed group algebra} of $G$ over $\mathcal{O}$ as:

\begin{equation}
\mathcal{O}G:=\varprojlim{\mathcal{O}[G/N]}
\end{equation}

\noindent where the limit is taken over all open normal subgroups $N$ of $G$. Continuous, $\mathcal{O}$-linear representations of $G$ are closely related to $\mathcal{O}G$-modules.\\ 

\noindent This paper is part of an ongoing project to classify the prime ideal structure of $\mathcal{O}G$, towards which much progress has been made in \cite{chevalley}, \cite{nilpotent}, \cite{munster}, \cite{APB} and \cite{primitive}. In the same vein as those works, we aim to prove that all prime ideals in $\mathcal{O}G$ can be reduced to a particular standard form. Specifically, recall the following definition \cite[Definition 1.1]{primitive}:

\begin{definition}\label{standard}

We say that a prime ideal $P$ of $\mathcal{O}G$ is \emph{standard} if there exists a closed, normal subgroup $H$ of $G$ such that:

\begin{itemize}

\item $G_0:=\frac{G}{H}$ is torsionfree.

\item $H-1\subseteq P$.

\item The image of $P$ in $\mathcal{O}G_0$ is centrally generated.

\end{itemize}

\noindent We say that $P$ is \emph{virtually standard} if $P\cap \mathcal{O}U$ is a finite intersection of standard prime ideals of $\mathcal{O}U$ for some open normal subgroup $U$ of $G$.

\end{definition}

\noindent The essence of this definition is that $P$ is standard when it can be constructed using only \emph{augmentation ideals} of the form $(H-1)\mathcal{O}G$, for $H$ a closed subgroup of $G$, and centrally generated ideals, i.e. the \emph{obvious} prime ideals.\\

\noindent In our case, we will assume further that $G$ is a \emph{uniform pro-$p$ group} in the sense of \cite[Definition 4.1]{DDMS}. This is a safe reduction since all compact $p$-adic Lie groups have an open, uniform normal subgroup.\\ 

\noindent Let us recall the main conjecture within the study of two-sided ideals in non-commutative Iwasawa algebras, first proposed in \cite[Question N]{survey}, and stated in \cite[Conjecture 1.1]{primitive}:

\begin{conjecture}\label{main}

Let $G$ be a solvable, uniform pro-$p$ group, and let $P$ be a prime ideal in $\mathcal{O}G$. Then $P$ is virtually standard, and moreover if $p\in P$ then $P$ is standard.

\end{conjecture}

\noindent\textbf{Note:} There is a version of this conjecture for non-solvable groups, which requires us to exclude the case where $\mathcal{O}G/P$ is a finitely generated $\mathcal{O}$-module, but this will not concern us here.\\

\noindent  When the prime ideal $P$ contains $p$, we can reduce to studying the \emph{mod-$p$ Iwasawa algebra}:

\begin{center}
$kG:=\frac{\mathcal{O}G}{(\pi)}=\varprojlim{k[G/N]}$.
\end{center} 

\noindent We know that Conjecture \ref{main} holds for all prime ideals $P$ of $kG$ whenever $G$ is nilpotent by \cite[Theorem A]{nilpotent}, and also when $G$ is \emph{abelian-by-procyclic} by \cite[Theorem 1.4]{APB}.\\

\noindent In the case where the prime ideal $P$ does not contain $p$, however, the picture is very different. Define the \emph{rational Iwasawa algebra} or \emph{Iwasawa algebra of continuous distributions} as 

\begin{center}
$KG:=\mathcal{O}G\otimes_{\mathcal{O}}K$.
\end{center} 

\noindent This is a Noetherian, topological $K$-algebra, and the prime ideals of $\mathcal{O}G$ not containing $p$ are in bijection with prime ideals in $KG$, via the map $P\mapsto P\otimes_{\mathcal{O}}K$. We aim to prove the analogue of Conjecture \ref{main} for prime ideals in $KG$.\\

\noindent\textbf{Note:} 1. This conjecture is trivially true for $G$ abelian.\\

\noindent 2. The requirement that prime ideals in $KG$ are only virtually standard is necessary, since they are \emph{not} all standard. For example, if $G=\mathbb{Z}_p\times\mathbb{Z}_p$, and $K$ contains a $p$'th root of unity $\zeta$, and we let $P$ be the kernel of the map $KG\to K\mathbb{Z}_p, (r,s)\mapsto\zeta^rs$. Then $P$ is a prime ideal of $KG$ and $P$ is not standard.\\

\noindent In this paper, we will prove a version of Conjecture \ref{main} for $KG$, in the case where $G$ is nilpotent.

\subsection{Alternative Formulation}

There is an alternative way of describing standard prime ideals in $\mathcal{O}G$ and $KG$, and thus formulating Conjecture \ref{main}, which will be of more practical use:\\

\noindent Firstly, for any two-sided ideal $I$ of $\mathcal{O}G$, recall from \cite[Definition 5.2]{nilpotent} that we define $I^{\dagger}:=\{g\in G:g-1\in I\}$, a closed, normal subgroup of $G$, and we say that $I$ is \emph{faithful} if $I^{\dagger}=1$, i.e. if the natural map $G\to\left(\frac{\mathcal{O}G}{I}\right)^{\times},g\mapsto g+I$ is injective.

Then setting $G_I:=\frac{G}{I}$, the kernel of the natural surjection $\mathcal{O}G\to\mathcal{O}G_I$ is the augmentation ideal $(I^{\dagger}-1)\mathcal{O}G$, and the image of $I$ under this surjection is a faithful ideal of $\mathcal{O}G_I$.\\

\noindent\textbf{Note:} If $I$ is prime and $p\in I$, it follows from \cite[Lemma 5.2]{nilpotent} then $G_I$ is torsionfree, but this need not be true if $p\notin I$. Roughly speaking, this is why we can only generally assert that prime ideals not containing $p$ are \emph{virtually standard}, and not standard.\\

\noindent If $P$ is a faithful, prime ideal of $\mathcal{O}G$, then to prove that $P$ is standard, we see using Definition \ref{standard} that it is only required to prove that $P$ is centrally generated. Using \cite[Corollary A]{centre}, we know that $Z(\mathcal{O}G)=\mathcal{O}Z(G)$, so $P$ is centrally generated precisely when $P=(P\cap\mathcal{O}Z(G))\mathcal{O}G$.\\

\noindent More generallly, if $I$ is a right ideal of $\mathcal{O}G$ and $H$ is a closed subgroup of $G$, we say that $H$ \emph{controls} $I$ if $I=(I\cap\mathcal{O}H)\mathcal{O}G$, i.e. $I$ is generated as a right ideal by a subset of $\mathcal{O}H$.  Define the \emph{controller subgroup} of $I$ by $I^{\chi}:=\bigcap\{U\leq_o G: U$ controls $I\}$, and it follows from \cite[Theorem A]{controller} that a closed subgroup $H$ of $G$ controls an ideal $I\trianglelefteq RG$ if and only if $I^{\chi}\subseteq H$, so in particular $I^{\chi}$ controls $I$.

If $I$ is a two-sided ideal, then $I^{\chi}$ is a closed, normal subgroup of $G$ by \cite[Lemma 5.3(a)]{nilpotent}, and to prove that $I$ is centrally generated, all that is required is to prove that $I$ is controlled by $Z(G)$, i.e. $I^{\chi}\subseteq Z(G)$.\\

\noindent So, to summarise, given a prime ideal $P$ of $\mathcal{O}G$, to prove that $P$ is standard, we need only to prove that the quotient $G_P=\frac{G}{P^{\dagger}}$ is torsionfree, and that the image of $P$ under the surjection $\mathcal{O}G\to\mathcal{O}G_P$ is controlled by $Z(G_P)$. Therefore, we deduce the following alternative formulation for Conjecture \ref{main}:\\

\noindent\textbf{Alternative Formulation:} Let $G$ be a solvable, uniform pro-$p$ group. We conjecture that every faithful prime ideal of $\mathcal{O}G$ is controlled by $Z(G)$.

\subsection{Main Results}

When studying Iwasawa algebras, rather than studying general prime ideals, we may be interested specifically in classifying \emph{primitive ideals}, i.e. the annihilators of simple $\mathcal{O}G$-modules.\\ 

\noindent However, since $G$ is a pro-$p$ group, the Iwasawa algebra $\mathcal{O}G$ has a unique maximal left ideal $\mathfrak{m}=(G-1,\pi)$, which is in fact two-sided. Therefore $\mathfrak{m}$ is the only primitive ideal in $\mathcal{O}G$. The rational Iwasawa algebra $KG$, on the other hand, has many simple modules and primitive ideals.

\begingroup
\setcounter{tmp}{\value{theo}}
\setcounter{theo}{0}
\renewcommand\thetheo{\Alph{theo}}

\begin{theo}\label{A}

Let $G$ be a nilpotent, uniform pro-$p$ group. Then every primitive ideal of $KG$ is maximal and virtually standard. Moreover, every faithful, primitive ideal of $KG$ is standard.

\end{theo}

\endgroup

\noindent As explained above, we see that to prove Theorem \ref{A}, it suffices to show that all faithful, primitive ideals in $KG$ are controlled by $Z(G)$.\\

\noindent Now, recall from \cite[\rom{3} 2.1.2]{Lazard} the definition of a \emph{$p$-valuation} $\omega:G\to\mathbb{R}\cup\{\infty\}$.  We will recap the key properties of $p$-valuations in section 2, but for now, just recall that if $G$ is uniform, then $G$ carries a complete $p$-valuation given by $\omega(g):=\sup\{n\in\mathbb{N}:g\in G^{p^{n+1}}\}$, so this concept gives rise to a larger class of torsionfree compact $p$-adic Lie groups which, in particular, contains the class of all closed subgroups of uniform groups.\\

\noindent If we assume that $(G,\omega)$ is a complete, nilpotent $p$-valued group of finite rank, then it follows from \cite[Theorem A]{nilpotent} that all faithful prime ideals in the mod-$p$ Iwasawa algebra $kG$ are controlled by $Z(G)$. One might think that these techniques could be generalised to the characteristic 0 case to prove the same result. Unfortunately, the author showed in \cite{primitive} that these techniques fail in characteristic 0, and they can only be used to establish a much weaker control theorem for primitive ideals (\cite[Theorem 1.2]{primitive}).

However, in this paper, we will adapt the argument given in \cite{primitive} with some new techniques, and prove the following much stronger control theorem for general prime ideals:

\begingroup
\setcounter{tmp}{\value{theo}}
\setcounter{theo}{1} 
\renewcommand\thetheo{\Alph{theo}}
\begin{theo}\label{C}

Let $G$ be a nilpotent, complete $p$-valued group of finite rank. Then there exists an abelian normal subgroup $A$ of $G$ such that $A$ controls every faithful prime ideal in $KG$.

\end{theo}

\endgroup

\noindent Of course, if we could show that this subgroup $A$ is central, then Theorem \ref{A} would follow immediately, and would remain true for prime ideals as opposed to just primitive ideals. But unfortunately, this need not always be the case. 

For example, if $G=H\rtimes\mathbb{Z}_p$ for $H$ abelian and $(G,H)\not\subseteq Z(G)$, then the subgroup $A$ given by Theorem \ref{C} is $H$, which is not central. We will prove Theorem \ref{C} in section 3.\\

\noindent Theorem \ref{C} is the strongest result we have obtained to date concerning general prime ideals in $KG$, but all subsequent results require the additional assumption that our prime ideals are \emph{primitive}.\\ 

\noindent The key idea is that we want to define a class of $KG$-representations $M$ whose annihilator ideals completely describe the primitive ideal structure of $KG$. Using  \cite[Theorem 5.2]{ST}, we have a dense, faithfully flat embedding fo $KG$ into the \emph{locally analytic distribution algebra} $D(G,K)$ as defined in \cite[Definition 2.1, Proposition 2.3]{ST'}, so it makes sense to restrict to the class of coadmissible $D(G,K)$-modules, which naturally have the structure of $KG$-modules. However, since $D(G,K)$ is non-noetherian, this may present difficulties, so instead we restrict our attention to larger, Noetherian completions of $KG$:\\

\noindent Returning to the case where $G$ is uniform, let $\mathcal{L}_G=\log(G)$ be the $\mathbb{Z}_p$-Lie algebra of $G$ as defined in \cite[Theorem 4.30]{DDMS}, and set $\mathcal{L}:=\frac{1}{p}\mathcal{L}_G$. Recall from from \cite[Definition 1.2]{aff-dix} that we define the \emph{affinoid enveloping algebra} of $\mathcal{L}$ with coefficients in $K$ to be:

\begin{equation}
\widehat{U(\mathcal{L})}_K:=\left(\underset{n\in\mathbb{N}}{\varprojlim}{U(\mathcal{L})/\pi^nU(\mathcal{L})}\right)\otimes_{\mathcal{O}}K
\end{equation}

\noindent This is a Noetherian, Banach $K$-algebra, and recall from \cite[Theorem 10.4]{annals} that there exists a continuous, dense embedding of $K$-algebras: 

\begin{equation}\label{embedding'}
KG\xhookrightarrow{}\widehat{U(\mathcal{L})}_K, g\mapsto \exp(\log(g)). 
\end{equation}

\noindent Unlike the embedding $KG\to D(G,K)$, this map is not faithfully flat, but we can still use it to study the representation theory of $KG$ via the representation theory of $\mathcal{L}$.\\

\noindent In section 2, we will recall from a previous work \cite{aff-dix} how we define the \emph{Dixmier module} $\widehat{D(\lambda)}$ of $\widehat{U(\mathcal{L})}_K$, corresponding to a linear form $\lambda\in$ Hom$_{\mathbb{Z}_p}(\mathcal{L},\mathcal{O})$. It follows from \cite[Theorem A]{aff-dix} that using the annihilators of these modules, we can completely describe the primitive ideal structure of $\widehat{U(\mathcal{L})}_K$. 

So now, we are interested in the restricted action of $KG$ on $\widehat{D(\lambda)}$, and the key result we need in the proof of Theorem \ref{A} is the following:

\begingroup
\setcounter{tmp}{\value{theo}}
\setcounter{theo}{2} 
\renewcommand\thetheo{\Alph{theo}}
\begin{theo}\label{B}

Let $G$ be a nilpotent, uniform pro-$p$ group such that $\mathcal{L}$ is powerful, let $F/K$ be a finite extension, and let $\lambda\in$ Hom$_{\mathbb{Z}_p}(\mathcal{L},\mathcal{O}_F)$ such that $\lambda|_{Z(\mathcal{L})}$ is injective.  Then $P:=$\emph{ Ann}$_{KG}\widehat{D(\lambda)}_F$ is controlled by $Z(G)$.

\end{theo}

\endgroup

\noindent\textbf{Note:} To say that $\mathcal{L}$ is powerful just means that $[\mathcal{L},\mathcal{L}]\subseteq p\mathcal{L}$.\\ 

\noindent We will prove this result in section 4. The key idea is that we know that the annihilator $P:=$ Ann$_{KG}\widehat{D(\lambda)}$ is controlled by an abelian normal subgroup $A$ of $G$ by Theorem \ref{C}, so we consider the action of $KA$ on $\widehat{D(\lambda)}$, and prove that the kernel of this action is controlled by $Z(G)$.

In section 5, we will apply \cite[Theorem A]{aff-dix}, to prove that it suffices to know that Dixmier annihilators are controlled by $Z(G)$ to establish the same result for all primitive ideals, and Theorem \ref{A} will follow immediately from Theorem \ref{B}.\\

\noindent\textbf{Acknowledgments:} I am very grateful to Konstantin Ardakov for many helpful comments. I would also like to thank EPSRC and the Heilbronn Institute for Mathematical Research for funding me in my research.

\section{Preliminaries}

\noindent\textbf{Notation:} For $g,h\in G$, we denote the group commutator by $(g,h):=ghg^{-1}h^{-1}$. Moreover, we will write $H\trianglelefteq_c^i G$ to mean that $H$ is a closed, \emph{isolated} normal subgroup of $G$, i.e. $\frac{G}{H}$ is torsionfree.

\subsection{Non-commutative valuations}

Let us first recap some basic notions of ring filtrations and valuations. Throughout, let $R$ be any ring.

\begin{definition}

A \emph{filtration} on $R$ is a map $w:R\to\mathbb{Z}\cup\{\infty\}$ such that $w(0)=\infty$ and for all $r,s\in R$:

\begin{itemize}

\item $w(r+s)\geq \min\{w(r),w(s)\}$.

\item $w(rs)\geq w(r)+w(s)$

\end{itemize}

\noindent We say that $w$ is \emph{separated} if $w(r)=\infty$ if and only if $r=0$, and $w$ is a \emph{valuation} if $w(rs)=w(r)+w(s)$ for all $r,s\in R$.

\end{definition}

\noindent For each $n\in\mathbb{Z}$, we define $F_nR:=\{r\in R:w(r)\geq n\}$, and define the \emph{associated graded ring} gr$_w$ $R$ to be $\underset{n\in\mathbb{Z}}{\oplus}{\frac{F_nR}{F_{n+1}R}}$ with multiplication $(r+F_{n+1}R)(s+F_{m+1}R)=rs+F_{n+m+1}R$. 

Note that $w$ is a valuation if and only if gr$_w$ $R$ is a domain.\\ 

\noindent If $r\in R$ and $w(r)=n$ then we denote gr$(r):=r+F_{n+1}R\in$ gr $R$.\\

\noindent Recall from \cite[Ch. \rom{2} Definition 2.2.1]{LVO} that a filtration $w:R\to\mathbb{Z}\cup\{\infty\}$ is \emph{Zariskian} if the \emph{Rees ring} $\tilde{R}:=\underset{n\in\mathbb{Z}}{\oplus}{F_nR}$ is Noetherian, and $F_1R\subseteq J(F_0R)$. We will not use this definition very often, but we will usually always assume that our filtrations are Zariskian.\\ 

\noindent Note that if $w$ is Zariskian, then it is separated and both $R$ and gr$_w$ $R$ are Noetherian, since they arise as quotients of the Rees ring.\\

\noindent\textbf{Example:} 1. If $R$ carries a filtration $w$, then the matrix ring $M_n(R)$ carries a filtration $w_n(A)=\min\{w(a_{i,j}):i,j=1,\cdots,n\}$ -- the \emph{standard matrix filtration}.\\

\noindent 2. If $I$ is a two-sided ideal of $R$ and $R$ carries a filtration $w$, then the quotient ring $\frac{R}{I}$ carries the \emph{quotient filtration} given by $\overline{w}(r+I)=\sup\{w(r+y):y\in I\}$. Note that gr$_{\overline{w}}$ $\frac{R}{I}=\frac{\text{gr}_w\text{ }R}{\text{gr }I}$, and if $w$ is Zariskian then $\overline{w}$ is Zariskian.\\

\noindent Now, recall the following definition (\cite[Definition 3.1]{APB})

\begin{definition}\label{non-comm-val'}

Let $Q$ be a simple artinian ring, and let $v:Q\to\mathbb{Z}\cup\{\infty\}$ be a filtration. We say that $v$ is a \emph{non-commutative valuation} if the completion $\widehat{Q}$ of $Q$ with respect to $v$ is isomorphic to a matrix ring $M_n(Q(D))$, where: 

\begin{itemize}

\item $Q(D)$ is the ring of quotients of some non-commutative DVR $D$ with uniformiser $\nu$,

\item the extension of $v$ to $\widehat{Q}$ is given by the standard matrix filtration corresponding to the $\nu$-adic filtration on $Q(D)$.

\end{itemize}

\end{definition}

\noindent Note that if $v$ is a non-commutative valuation on $Q$, then for all $z\in Z(Q)$, $q\in Q$, $v(qz)=v(q)+v(z)$, a property which will be very useful to us in section 3.\\

\noindent The following construction allows us to define a non-commutative valuation on the artinian ring of quotients $Q(R)$ of a Zariskian filtered ring $R$. This construction was derived in \cite[Section 3]{nilpotent}, and we state it fully since we will need it for some proofs in section 3.

\begin{construction}\label{non-comm-val}

Let $R$ be a prime, Noetherian ring with a Zariskian filtration $w$ such that gr$_w$ $R$ is commutative and the graded ideal $($gr$_w$ $R)_{\geq 0}$ is non-nilpotent. Then for each minimal prime ideal $\mathfrak{q}$ of gr$_w$ $R$, we can construct a non-commutative valuation on $Q(R)$ using the following data:

\begin{itemize}

\item $S:=\{r\in R:$\emph{ gr}$(r)\notin\mathfrak{q}\}$ -- an Ore set in $R$ such that $S^{-1}R=Q(R)$

\item $w'$ -- a Zariskian filtration on $Q(R)$ such that $w'(r)\geq w(r)$ for all $r\in R$, and $w'(s^{-1}r)=w'(r)-w(s)$ for all $s\in S$. The associated graded gr$_{w'}$ $Q(R)$ is the homogeneous localisation of gr$_w$ $R$ at $\mathfrak{q}$.

\item $Q'$ -- the completion of $Q(R)$ with respect to $w'$, an artinian ring.

\item $U$ -- the positive part of $Q'$, a Noetherian ring.

\item $z$ -- a regular, normal element of $J(U)$ such that $z^nU=F_{nw'(z)}Q'$ for all $n\in\mathbb{Z}$.

\item $v_{z,U}$ -- the $z$-adic filtration on $Q'$, topologically equivalent to $w'$.

\item $\widehat{Q}$ -- a simple quotient of $Q'$.

\item $V$ -- the image of $U$ in $\widehat{Q}$.

\item $\overline{z}$ -- the image of $z$ in $V$.

\item $v_{\overline{z},V}$ -- the $\overline{z}$-adic filtration on $\widehat{Q}$.

\item $\mathcal{B}$ -- a maximal order in $\widehat{Q}$, equivalent to $V$, satisfying $\mathcal{B}\subseteq \overline{z}^{-r}V$ for some $r\geq 0$, isomorphic to $M_n(D)$ for some non-commutative DVR $D$.

\item $v_{\overline{z},\mathcal{B}}$ -- the $\overline{z}$-adic filtration on $\mathcal{B}$.

\item $v_{\mathfrak{q}}$ -- the $J(\mathcal{B})$-adic filtration on $\widehat{Q}$, topologically equivalent to $v_{\overline{z},\mathcal{B}}$.

\end{itemize}

\noindent Then $v=v_{\mathfrak{q}}$ defines a non-commutative valuation on $Q(R)$, whose completion is $\widehat{Q}$, and the natural map $R\to Q(R)$ is continuous. Moreover, if $w(x)\geq 0$ then $v(x)\geq 0$.

\end{construction}

\subsection{Crossed Products}

Given a ring $R$ and a group $H$, recall from \cite{Passman} that a \emph{crossed product} of $R$ with $H$, denoted $R\ast H$, is a ring extension $R\subseteq S$, free as a left $R$-module with basis $\{\overline{h}:h\in H\}\subseteq S^{\times}$ in bijection with $H$ such that for each $g,h\in H$:

\begin{itemize}

\item $\overline{g}R=R\overline{g}$ and \item $\overline{g}R\overline{h}R=\overline{gh}R$.

\end{itemize}

\noindent Furthermore, given a sequence $H_1,\cdots,H_r$ of groups, we denote an \emph{iterated crossed product} $R\ast H_1\ast H_2\ast\cdots\ast H_r$ inductively to mean a crossed product of $R\ast H_1\ast\cdots\ast H_{r-1}$ with $H_r$.\\

\noindent Let us recap some properties of crossed products that we will use throughout.

\begin{lemma}\label{semiprime}

Let $R$ be a Noetherian $\mathbb{Q}$-algebra, $F$ a finite group. Then if $P$ is a prime ideal of a crossed product $S=R\ast F$, then:

$i$. $P\cap R$ is semiprime in $R$.

$ii$. $J:=(P\cap R)\cdot S$ is semiprime in $S$, and $P$ is a minimal prime above $J$.

$iii$. $S/J=(P/P\cap R)\ast F$.

\end{lemma}

\begin{proof}

We will prove that $P\cap R$ is an $F$-prime ideal, i.e. it is $F$-invariant and for any $F$-invariant ideals $A,B$ of $R$, if $AB\subseteq P\cap R$ then $A\subseteq P\cap R$ or $B\subseteq P\cap R$.\\ 

\noindent Having established this, part $i$ follows from the fact that all minimal primes above $P\cap R$ form a single $F$-orbit by \cite[Lemma 14.2($ii$)]{Passman}, part $iii$ is obvious since $J=\underset{g\in F}{\oplus}(P\cap R)\bar{g}$, and part $ii$ is part $iii$ together with \cite[Proposition 10.5.8]{McConnell} and \cite[Theorem 4.4]{Passman}.\\

\noindent So, suppose $A,B\trianglelefteq R$ are $F$-invariant, i.e. for all $g\in F$, $\bar{g}A=A\bar{g}$ and $\bar{g}B=B\bar{g}$, and suppose that $AB\subseteq P\cap R$. Then $AS,BS$ are two-sided ideals of $S$, and $(AS)(BS)\subseteq P$. So since $P$ is prime, we can assume without loss of generality that $AS\subseteq P$.

So since $AS=\underset{g\in F}{\oplus}{A\bar{g}}$, it follows that $A\subseteq P\cap R$, and hence $P\cap R$ is $F$-prime as required.\end{proof}

\begin{lemma}\label{semiprimitive}

Let $R$ be a Noetherian ring, $F$ a finite group. Then if $P$ is a primitive ideal of a crossed product $R\ast F$, then $P\cap R$ is semiprimitive.

\end{lemma}

\begin{proof}

Let $S=R\ast F$, then since $P$ is primitive, $P=\text{ Ann}_S M$ for some irreducible $S$-module $M$. Since $F$ is finite, $M$ is finitely generated over $R$, so since $R$ is Noetherian, we can choose a maximal $R$-submodule $U$ of $M$.\\

\noindent For each $g\in F$, $gUg^{-1}$ is a maximal $R$-submodule of $M$, so set $M_g:=M/gUg^{-1}$, an irreducible $R$-module, and let $Q_g:=Ann_R M_g$, a primitive ideal of $R$. Clearly if $r\in P\cap R=Ann_RM$ then $rN_g=0$ for all $g\in F$, so $P\cap R\subseteq\underset{g\in F}{\cap}{Q_g}$.\\

Also, $\underset{g\in F}{\cap}{gUg^{-1}}$ is an $S$-submodule, so by simplicity of $M$, $\underset{g\in F}{\cap}{gUg^{-1}}=0$. So if $r\in\underset{g\in F}{\cap}{Q_g}$ then $rM_g=0$ for all $g$, so $rM\subseteq gUg^{-1}$ for all $g$, i.e. $rM\subseteq\underset{g\in F}{\cap}{gUg^{-1}}=0$ and hence $r\in Ann_RM=P\cap R$. Hence:

\begin{center}
$P\cap R=\underset{g\in F}{\cap}{Q_g}$
\end{center}

\noindent Hence $P\cap R$ is semiprimitive as required.\end{proof}

\subsection{$p$-valued Groups}

Let $G$ be a group. Recall from \cite[\rom{3} 2.1.2]{Lazard} that we define a \emph{$p$-valuation} on $G$ to be a map $\omega:G\to\mathbb{R}\cup\{\infty\}$ such that for all $g,h\in G$:

\begin{itemize}

\item $\omega(g)=\infty$ if and only if $g=1$.

\item $\omega(g^{-1}h)\geq\min\{\omega(g),\omega(h)\}$.

\item $\omega((g,h))\geq \omega(g)+\omega(h)$.

\item $\omega(g^p)=\omega(g)+1$.

\item $\omega(g)>\frac{1}{p-1}$.

\end{itemize}

\noindent Note that if $(G,\omega)$ is a $p$-valued group then $G$ is torsionfree, and carries a topology defined by the metric $d(g,h):=c^{-\omega(g^{-1}h)}$ for some $c>1$. We will always assume that $G$ is complete with respect to this topology, in which case we can define \emph{$p$-adic exponentiation} in $G$, i.e. for all $g\in G$, $\alpha\in\mathbb{Z}_p$, if $\alpha=\underset{n\rightarrow\infty}{\lim}{\alpha_i}$ for $\alpha_i\in\mathbb{Z}$, we define $g^{\alpha}:=\underset{n\rightarrow\infty}{\lim}{g^{\alpha_i}}\in G$.\\

\noindent Given $d\in\mathbb{N}$, we say that $G$ has \emph{finite rank $d$} if there exists a subset $\underline{g}:=\{g_1,\cdots,g_d\}\subseteq G$ such that for every $g\in G$, there exists a unique $\alpha\in\mathbb{Z}_p^d$ such that $g=\underline{g}^{\alpha}:=g_1^{\alpha_1}\cdots g_d^{\alpha_d}$, and $\omega(g):=\min\{v_p(\alpha_i)+\omega(g_i):i=1,\cdots,d\}$. We call such a subset $\underline{g}$ an \emph{ordered basis} for $(G,\omega)$.\\

\noindent\textbf{Example:} If $G$ is uniform, and $\omega(g):=\sup\{n\in\mathbb{N}:g\in G^{p^{n+1}}\}$, then $(G,\omega)$ is a complete $p$-valued group, and any minimal topological generating set for $G$ is an ordered basis for $(G,\omega)$.

\begin{definition}\label{abelian-val}

We say that a $p$-valuation $\omega:G\to\mathbb{R}\cup\{\infty\}$ is \emph{abelian} if:

\begin{itemize}

\item There exists $n\in\mathbb{N}$ such that $\omega(G)\subseteq\frac{1}{n}\mathbb{Z}$.

\item For all $g,h\in G$, $\omega((g,h))>\omega(g)+\omega(h)$.

\end{itemize}

\end{definition}

\noindent Using \cite[Lemma 26.13]{Schneider}, if $(G,\omega)$ is any integer valued $p$-valued group of finite rank (e.g. a uniform group), then we can choose $c>0$ such that $\omega_c(g):=\omega(g)-c$ is an abelian $p$-valuation on $G$.\\

\noindent Now, suppose that $(G,\omega)$ is a complete $p$-valued group of rank $d$, with ordered basis $\underline{g}=\{g_1,\cdots,g_d\}$ then the Iwasawa algebra $\mathcal{O}G$ is isomorphic to the power series ring $\mathcal{O}[[b_1,\cdots,b_d]]$ as an $\mathcal{O}$-module (and as a ring if $G$ is abelian), where each variable $b_i$ corresponds with $g_i-1$.\\

\noindent Moreover, if $\omega$ is an abelian $p$-valuation, taking values in $\frac{1}{n}\mathbb{Z}$ for some $n\in\mathbb{N}$, then recall from \cite[Section 2.2]{primitive} that we can define a filtration $w$ on $\mathcal{O}G$ via: 

\begin{center}
$w(\underset{\alpha\in\mathbb{N}^d}{\sum}{\lambda_{\alpha}b_1^{\alpha_1}\cdots b_d^{\alpha_d}})=\inf\{v_{\pi}(\lambda_{\alpha})+\underset{i\leq d}{\sum}{en\alpha_i\omega(g_i)}:\alpha\in\mathbb{N}^d\}$,
\end{center} 

\noindent where $e$ is the ramification index of $K/\mathbb{Q}_p$.\\ 

\noindent We call $w$ the \emph{Lazard filtration} on $\mathcal{O}G$. Using \cite[Theorem 4.5]{ST}, we see that gr$_w$ $\mathcal{O}G\cong k[t,t_1,\cdots,t_d]$, where $k$ is the residue field of $K$, $t=$ gr$(\pi)$ and $t_i=$ gr$(b_i)$, and hence is commutative. Note that for any $g\in G$, $w(g-1)\geq en\omega(g)$, with equality if $g=g_i$ for some $i$.\\

\noindent Furthermore, since $\mathcal{O}G$ is complete with respect to $w$ and gr$_w$ $\mathcal{O}G$ is Noetherian, it follows from \cite[Ch. \rom{2} Theorem 2.1.2]{LVO} that $w$ is a Zariskian filtration. Hence for any two-sided ideal $I$ of $\mathcal{O}G$, the quotient filtration $\overline{w}$ on $\mathcal{O}G/I$ is Zariskian.

In particular, if $I$ is a prime ideal then we can use Construction \ref{non-comm-val} to define a non-commutative valuation $v$ on the Goldie ring of quotients $Q(\mathcal{O}G/I)$ such that the natural map $\tau:(\mathcal{O}G,w)\to (Q(\mathcal{O}G/I),v)$ is continuous.

\subsection{Prime ideals in $KG$}

Fixing $(G,\omega)$ a complete $p$-valued group of finite rank, we will now examine some basic properties of prime ideals in $KG$. First of all, the following lemma allows us to simplify the statement of Theorem \ref{B} to remove reference to the finite extension $F/K$:

\begin{lemma}\label{base-change}

Let $F/K$ be a finite extension, and let $I'$ a right ideal of $FG$. Setting $I:=I'\cap KG$, we have that if $I'$ is controlled by $U\leq_c G$ then $I$ is controlled by $U$.

\end{lemma}

\begin{proof}

\noindent We will first suppose that $U$ is open in $G$. Then given $r\in I$, choose a complete set of coset representatives $\{g_1,\cdots,g_r\}$ for $U$ in $G$, then $r=\underset{1\leq i\leq r}{\sum}{r_{i}g_i}$ for some $r_{i}\in KU\subseteq FU$.

So since $I=I'\cap KG$ and $I'$ is controlled by $U$, it follows that $r_{i}\in I'\cap FU\cap KG=I'\cap KU=I$ for each $i$, and hence $I$ is controlled by $U$.\\

\noindent So, let $I^{\chi}$ be the controller subgroup of $I$, i.e. the intersection of all open subgroups of $G$ controlling $I$. So since this includes all open subgroups of $G$ controlling $I'$, we have that $I^{\chi}\subseteq I'^{\chi}$, hence any closed subgroup controlling $I'$ also controls $I$.\end{proof}

\vspace{0.1in}
 
\noindent Now, recall that a two-sided ideal $P$ of a ring $R$ is \emph{completely prime} if the quotient $\frac{R}{P}$ is a domain. The following result is the characteristic 0 analogue of \cite[Theorem 8.6]{nilpotent}, and it uses a similar argument.

\begin{theorem}\label{completely-prime'}

Let $P$ be a prime ideal of $KZ(G)$. Then $PKG$ is a completely prime ideal of $KG$, and if $P$ is faithful then $PKG$ is faithful.

\end{theorem}

\begin{proof}

Let $Z:=Z(G)$. We will prove that if $P$ is a prime ideal of $\mathcal{O}Z$ with $p\notin P$ then $P\mathcal{O}G$ is completely prime, and it is faithful if $P$ is faithful. The result for the rational Iwasawa algebras follows immediately.\\ 

\noindent Let $Q$ be the field of fractions of $\mathcal{O}Z/P$. If we let $w$ be the Lazard filtration on $\mathcal{O}Z$, then since $w$ is a Zariskian filtration and the associated graded is a commutative, infinite dimensional $k$-algebra, it follows from Construction \ref{non-comm-val} that there exists a valuation $v'$ on $Q$ such that the natural map $\tau:\mathcal{O}G\to Q$ is continuous, and if $w(x)\geq 0$ then $v'(\tau(x))\geq 0$.\\

\noindent Furthermore, if $v'(\tau(z-1))=0$ for some $z\in Z$ then $v'(\tau(z-1)^n)=0$ for all $n$ since $v'$ is a valuation, which is a contradiction since $(z-1)^n$ converges to zero in $\mathcal{O}G$, and hence in $Q$ by continuity of $\tau$. Therefore $v'(\tau(z-1))>0$ for all $z\in Z(G)$, and after choosing an ordered basis $\{z_1,\cdots,z_n\}$ for $Z$ and an integer $M$ such that $Mv'(\tau(z_i-1))\geq w(z_i-1)$ for all $i$, then we obtain an equivalent valuation $v:=Mv'$ on $Q$ such that $v(\tau(x))\geq w(x)$ for all $x\in\mathcal{O}Z$.\\

\noindent Recall that if we fix an ordered basis $\{g_1,\cdots,g_e\}$ for $\frac{G}{Z}$, then every element of $\mathcal{O}G$ has the form $\underset{\alpha\in\mathbb{N}^e}{\sum}{\mu_{\alpha}\underline{c}^{\alpha}}$ for some $\mu_{\alpha}\in\mathcal{O}Z$ where $c_i=g_i-1$. Define a map $u:\mathcal{O}G\to\mathbb{Z}\cup\{\infty\}$ via:

\begin{equation}
u:\mathcal{O}G\to\mathbb{Z}\cup\{\infty\},\underset{\alpha\in\mathbb{N}^e}{\sum}{\mu_{\alpha}\underline{c}^{\alpha}}\mapsto \inf\{v(\tau(\mu_{\alpha}))+w(\underline{c}^{\alpha}):\alpha\in\mathbb{N}^e\}.
\end{equation}

\noindent Since $v$ is a separated valuation, it is clear that $u(\underset{\alpha\in\mathbb{N}^e}{\sum}{\mu_{\alpha}\underline{c}^{\alpha}})=\infty$ if and only if $\mu_{\alpha}\in P$ for all $\alpha$, i.e. if and only if $\underset{\alpha\in\mathbb{N}^e}{\sum}{\mu_{\alpha}\underline{c}^{\alpha}}\in P\mathcal{O}G$. Therefore $u^{-1}(\infty)=P\mathcal{O}G$. So following the proof of \cite[Theorem 8.6]{nilpotent}, we will prove that $u$ is a valuation on $\mathcal{O}G$, from which it will follow that $P\mathcal{O}G=u^{-1}(\infty)$ is a completely prime ideal.\\

\noindent Firstly, it is clear from the definition that $u(r+s)\geq\min\{u(r),u(s)\}$, $u(\mu)=v(\tau(\mu))$ and $u(\mu r)=u(\mu)+u(r)$ for all $r,s\in\mathcal{O}G$, $\mu\in\mathcal{O}Z$. It is also clear that if $r_1,r_2,\cdots\in\mathcal{O}G$ with $r_i\rightarrow 0$ as $i\rightarrow\infty$ then $u(r_1+r_2+\cdots)\geq\inf\{u(r_i):i\geq 1\}$, therefore to prove that $u$ is a filtration it remains to prove that $u(\underline{c}^{\alpha}\underline{c}^{\beta})\geq u(\underline{c}^{\alpha})+ u(\underline(c)^{\beta})$ for all $\alpha,\beta\in\mathbb{N}^r$.\\

\noindent Write $\underline{c}^{\alpha}\underline{c}^{\beta}=\underset{\gamma\in\mathbb{N}^e}{\sum}{\lambda_{\gamma}^{\alpha,\beta}\underline{c}^{\gamma}}$, then by the definition of the Lazard filtration, $w(\underset{\gamma\in\mathbb{N}^e}{\sum}{\lambda_{\gamma}^{\alpha,\beta}\underline{c}^{\gamma}})=\inf\{w(\lambda_{\gamma}^{\alpha,\beta})+w(\underline{c}^{\gamma}):\gamma\in\mathbb{N}^d\}$. So since $u(x)\geq w(\tau(x))$ for all $x\in\mathcal{O}Z$, we have:

\begin{center}
$u(\underline{c}^{\alpha}\underline{c}^{\beta})=\inf\{v(\tau(\lambda_{\gamma}^{\alpha,\beta}))+w(\underline{c}^{\gamma}):\gamma\in\mathbb{N}^e\}\geq\inf\{w(\lambda_{\gamma}^{\alpha,\beta})+w(\underline{c}^{\gamma}):\gamma\in\mathbb{N}^e\}=w(\underline{c}^{\alpha})+w(\underline{c}^{\beta})=u(\underline{c}^{\alpha})+u(\underline{c}^{\beta})$.
\end{center}

So $u$ is a filtration on $\mathcal{O}G$, and to verify that it is a valuation, we will show that the associated graded gr$_u$ $\mathcal{O}G$ is a domain. First note that the definition of $u$ gives rise to a natural inclusion of graded rings gr$_{v}$ $\mathcal{O}Z/P\to$ gr$_u$ $\mathcal{O}G$, and this gives rise to an isomorphism of graded rings gr$_v$ $(\mathcal{O}Z/P)[Y_1,\cdots,Y_e]\to$ gr$_u$ $\mathcal{O}G$ where $Y_i$ is sent to gr$(c_i)$. Therefore gr$_u$ $\mathcal{O}G$ is a domain and $u$ is a valuation as required.\\

\noindent Finally, if $P$ is faithful, then suppose $g\in G$ and $g-1\in P\mathcal{O}G$. Then write $g=zg_1^{\alpha_1}\cdots g_e^{\alpha_e}$ for some $z\in Z$, $\alpha_i\in\mathbb{Z}_p$, and it follows that:

\begin{center}
$h-1=(z-1)+(z-1)\underset{0\neq\gamma\in\mathbb{N}^e}{\sum}{\binom{\alpha}{\gamma}\underline{c}^{\alpha}}+\underset{0\neq\gamma\in\mathbb{N}^e}{\sum}{\binom{\alpha}{\gamma}\underline{c}^{\alpha}}$.
\end{center}

\noindent Therefore, we see that $z-1\in P$ and hence $z=1$ since $P$ is faithful. It also follows that for each $0\neq\gamma\in\mathbb{N}^e$, $\binom{\alpha}{\gamma}\in P$, and hence $\binom{\alpha}{\gamma}=0$ since $P\cap\mathcal{O}=0$. This is only possible if $\alpha=(\alpha_1,\cdots,\alpha_e)=0$, and hence $h=zg_1^{\alpha_1}\cdots g_e^{\alpha_e}=1$ and $P\mathcal{O}G$ is faithful as we require.\end{proof}

\noindent In particular, it follows from this result that standard prime ideals in $KG$ are completely prime.

\subsection{Completions of $KG$}\label{dist(G)}

For the rest of this section, fix $G$ a uniform pro-$p$ group, let $\mathcal{L}:=\frac{1}{p}\log(G)$ be its $\mathbb{Z}_p$-Lie algebra of $G$, and let $\mathfrak{g}=\mathcal{L}\otimes_{\mathbb{Z}_p}\mathbb{Q}_p$.\\

\noindent To reiterate, we aim to study the action of $KG$ on certain $\widehat{U(\mathcal{L})}_K$-modules using the dense embedding $KG\to\widehat{U(\mathcal{L})}_K$. However, this embedding is not faithfully flat, so representation theoretic information is lost when passing from $KG$ to $\widehat{U(\mathcal{L})}_K$.

Perhaps a better choice for a completion of $KG$ would be the \emph{distribution algebra} $D(G,K)$ of $G$ with coefficients in $K$ in the sense of \cite{ST}. In this case, the natural dense embedding $KG\to D(G,K)$ is faithfully flat by \cite[Theorem 4.11]{ST}, but unfortunately $D(G,K)$ is not Noetherian, so it would be difficult in practice to extract general ring-theoretic information from $D(G,K)$.\\

\noindent However, for each $n\geq 0$, consider the crossed products 

\begin{center}
$D_{p^n}=D_{p^n}(G):=\widehat{U(p^n\mathcal{L})}_K\ast\frac{G}{G^{p^n}}$
\end{center} 

\noindent as defined in \cite[Proposition 10.6]{annals}, which arise as a Banach completions of $KG$ with respect to the extension of the dense embedding $KG^{p^n}\to\widehat{U(p^n\mathcal{L})}_K$ to $KG=KG^{p^n}\ast\frac{G}{G^{p^n}}$. These algebras give rise to an inverse system:

\begin{center}
$KG\to D(G,K)\to\cdots D_{p^3}\to D_{p^2}\to D_p\to D_0=\widehat{U(\mathcal{L})}_K$.
\end{center}

\noindent i.e. $D(G,K)=\underset{n\rightarrow\infty}{\varprojlim}{D_{p^n}}$, so since $D(G,K)$ is faithfully flat over $KG$, we want to approximate $D(G,K)$ using the Noetherian Banach algebras $D_{p^n}$, and thus limit how much information we lose. Indeed, using \cite[Proposition 10.6(e), Corollary 10.11]{annals}, we see that for all $KG$-modules $M$, $D_{p^n}\otimes_{KG}M\neq 0$ for all sufficiently high $n$.

\begin{lemma}\label{cyclic}

Let $A$ be a free abelian pro-$p$ group of rank $d$, $\mathcal{A}:=\frac{1}{p}\log(A)$. Then $\frac{A}{A^p}=C_1\times\cdots\times C_d$ where each $C_i=\langle c_i\rangle=\langle g_iA^p\rangle$ is a cyclic group of order $p$, and $D_p=D_p(A)$ is an iterated crossed product:

\begin{center}
$D_p=\widehat{U(p\mathcal{A})}_K\ast C_1\ast\cdots\ast C_d$.
\end{center}

\noindent where for each $i=1,\cdots,d$ $\overline{c_i}^r=\overline{c_i^r}$ for $0\leq r<p$, and $\overline{c_i}^p=g_i^p$.

\end{lemma}

\begin{proof}

Firstly, it is clear that since $A=\mathbb{Z}_p^d$ that $\frac{A}{A^p}=\frac{\mathbb{Z}_p^d}{(p\mathbb{Z}_p)^d}=(\frac{\mathbb{Z}_p}{p\mathbb{Z}_p})^d=C_1\times\cdots\times C_d$ as required.\\

\noindent For the second statement, it suffices to prove that $KA=KA^p\ast C_1\ast\cdots\ast C_d$, and that this decomposition satisfies the same properties, since it will be preserved after passing to the completion.\\

\noindent Choose a $\mathbb{Z}_p$-basis $\{g_1,\cdots,g_d\}$ for $A$, and we may assume that $C_i=\langle c_i\rangle$ where $c_i=g_iA^p$. Then every element $r\in KA$ has the form $\underset{\alpha\in [p-1]^d}{\sum}{r_{\alpha}g_1^{\alpha_1}\cdots g_d^{\alpha_d}}$ for some $r_{\alpha}\in KA^p$, and $r$ is sent to $\underset{\alpha\in [p-1]^d}{\sum}{r_{\alpha}\overline{c_1^{\alpha_1}\cdots c_d^{\alpha_d}}}$ under the isomorphism $KA\to KA^{p}\ast\frac{A}{A^p}$.\\

\noindent So, since $g_i$, $g_j$ and $g_ig_j$ are sent to $\overline{c_i}$, $\overline{c_j}$ and $\overline{c_ic_j}$ respectively, it follows that $\overline{c_ic_j}=\overline{c_i}\cdot\overline{c_j}$ for each $i,j$. Hence $KA=KA^p\ast C_1\ast\cdots\ast C_d$.\\

\noindent Finally, for $0\leq r<p$, $g_i^r$ is sent to $\overline{c_i^r}$, and hence $\overline{c_i}^r=\overline{c_i^r}$, and $g_i^p\in KA^p$ is sent to $g_i^p$, so $\overline{c_i}^p=g_i^p$ as required. \end{proof}

\subsection{Dixmier modules}

Recall from \cite[Definition 2.3]{aff-dix} the following definition:

\begin{definition}

Let $\lambda:\mathfrak{g}\to K$ be a $\mathbb{Q}_p$-linear form such that $\lambda(\mathcal{L})\subseteq\mathcal{O}$ (i.e. $\lambda\in$\emph{ Hom}$_{\mathbb{Z}_p}(\mathcal{L},\mathcal{O})$).

\begin{itemize}

\item A \emph{polarisation} of $\mathfrak{g}_K=\mathfrak{g}\otimes_{\mathbb{Q}_p}K$ at $\lambda$ is a solvable subalgebra $\mathfrak{b}$ of $\mathfrak{g}_K$ such that for any subspace $\mathfrak{b}\subseteq V\subseteq\mathfrak{g}_K$, $\lambda([V,V])=0$ if and only if $V=\mathfrak{b}$.

\item Given a polarisation $\mathfrak{b}$ of $\mathfrak{g}_K$ at $\lambda$, let $\mathcal{B}:=\mathfrak{b}\cap(\mathcal{L}\otimes_{\mathbb{Z}_p}\mathcal{O})$, and let $K_{\lambda}$ be the one-dimensional $\mathfrak{b}$-module induced by $\lambda$. Define the \emph{affinoid Dixmier module} of $\widehat{U(\mathcal{L})}_K$ induced by $\lambda$ to be $\widehat{D(\lambda)}=\widehat{D(\lambda)}_{\mathfrak{b}}:=\widehat{U(\mathcal{L})}_K\otimes_{\widehat{U(\mathcal{B})}_K}K_{\lambda}$

\end{itemize}

\end{definition}

\noindent\textbf{Note:} If it is unclear what the base field $K$ is, we may sometimes write $\widehat{D(\lambda)}_K$.\\

\noindent So, fixing $\lambda\in$ Hom$_{\mathbb{Z}_p}(\mathcal{L},\mathcal{O})$, let $\mathfrak{b}$ be a polarisarion of $\mathfrak{g}$ at $\lambda$, and we see that $KG$ acts on $\widehat{D(\lambda)}_{\mathfrak{b}}$ via the embedding $KG\to\widehat{U(\mathcal{L})}_K$. Set $P:=$ Ann$_{KG}\widehat{D(\lambda)}$, and using \cite[Theorem 4.4]{aff-dix}, we see that this does not depend on the choice of polarisation.

\begin{definition}\label{l-scalar}

Define the \emph{$\lambda$-scalar ideal} of $\mathfrak{g}$ to be $\mathfrak{a}_{\lambda}$, the largest ideal of $\mathfrak{g}$ such that $\lambda(\mathfrak{a})=0$. Also, set $\mathcal{A}_{\lambda}:=\mathfrak{a}_{\lambda}\cap\mathcal{L}$, and define the $\lambda$-scalar subgroup of $G$, $A_{\lambda}:=\exp(p\mathcal{A})\trianglelefteq_c^i G$.

\end{definition}

\noindent\textbf{Note:} For any choice of polarisation $\mathfrak{b}$ of $\mathfrak{g}\otimes_{\mathbb{Q}_p}K$ at $\lambda$, it follows from \cite[Lemma 2.3]{aff-dix} that $\mathfrak{a}_{\lambda}\subseteq\mathfrak{b}$.

\begin{lemma}\label{dagger}

Let $A_{\lambda}$ be the $\lambda$-scalar subgroup of $G$. Then if $P=Ann_{KG}\widehat{D(\lambda)}$, then $A_{\lambda}=P^{\dagger}=\{g\in G:g-1\in P\}$. In particular, $P$ is faithful if and only if the restriction of $\lambda$ to $Z(\mathfrak{g})$ is injective.

\end{lemma}

\begin{proof}

Firstly, since $\mathfrak{a}_{\lambda}\subseteq\mathfrak{b}$, we see that $\mathfrak{a}_{\lambda}\widehat{D(\lambda)}=\mathfrak{a}_{\lambda}\widehat{U(\mathcal{L})}_K\otimes_{\widehat{U(\mathcal{B})}_K}K_{\lambda}=\widehat{U(\mathcal{L})}_K\mathfrak{a}_{\lambda}\otimes_{\widehat{U(\mathcal{B})}_K}K_{\lambda}=0$.\\

\noindent So since $A_{\lambda}-1\subseteq\mathfrak{a}_{\lambda}\widehat{U(\mathcal{L})}_K$, it is clear that $A_{\lambda}-1\subseteq P$, i.e. $A_{\lambda}\subseteq P^{\dagger}$.\\

\noindent Now, since $T=P^{\dagger}$ is a closed, normal subgroup of $G$, $\mathcal{T}:=\frac{1}{p}\log(T)$ is an ideal of $\mathcal{L}$, and it contains $\frac{1}{p}\log(A_{\lambda})=\mathcal{A}_{\lambda}$. Also, since $(T-1)\widehat{D(\lambda)}=0$, it follows that $\mathcal{T}\widehat{D(\lambda)}=0$. This is only possible if $\mathcal{T}\subseteq\mathcal{B}$ and $\lambda(\mathcal{T})=0$.\\

\noindent Setting $\mathfrak{t}:=\mathcal{T}\otimes_{\mathbb{Z}_p}\mathbb{Q}_p$, $\mathfrak{t}$ is an ideal of $\mathfrak{g}$, $\mathfrak{a}_{\lambda}\subseteq\mathfrak{t}$ and $\lambda(\mathfrak{t})=0$. So by the definition of $\mathfrak{a}_{\lambda}$, this means that $\mathfrak{a}_{\lambda}=\mathfrak{t}$.\\

\noindent So, for any $u\in\mathcal{T}$, there exists $i\in\mathbb{N}$ such that $\pi^iu\in\mathcal{A}_{\lambda}=\mathfrak{a}_{\lambda}\cap\mathcal{L}$, and this means that $u\in\mathcal{A}_{\lambda}$. So $\mathcal{A}_{\lambda}=\mathcal{T}$, and it follows immediately that $A_{\lambda}=T$.\\

\noindent Finally, since $G$ is nilpotent, $\mathcal{L}$ is nilpotent, and thus if $\mathcal{A}_{\lambda}\neq 0$, then it must have non-trivial intersection with $Z(\mathfrak{g})$. So since $P$ is faithful if and only if $A_{\lambda}=1$ (i.e. if and only if $\mathfrak{a}_{\lambda}=0$), and any subspace of $Z(\mathfrak{g})$ is an ideal of $\mathfrak{g}$, it follows that $P$ is faithful precisely when nothing in $Z(\mathfrak{g})$ is sent to zero under $\lambda$, i.e. $\lambda|_{Z(\mathfrak{g})}$ is injective.\end{proof}

This lemma is useful to know, because it implies that for any Dixmier annihilator $P$, $P^{\dagger}$ is a closed, isolated normal subgroup of $G$, and hence we can replace $G$ by $G_P=\frac{G}{P^{\dagger}}$, which is still a nilpotent, uniform group, and $P_0=\frac{P}{(P^{\dagger}-1)KG}$ becomes a faithful Dixmier annihilator.\\

\noindent Note that this lemma explains why we need the assumption that $\lambda|_{Z(\mathfrak{g})}$ is injective in the statement of Theorem \ref{B}, since it is generally untrue that non-faithful prime ideals in $KG$ are controlled by $Z(G)$.

\section{The Logarithm of Automorphisms}

In this section, we will study general prime ideals within the rational Iwasawa algebra $KG$ of a $p$-valuable group $G$. This is of course equivalent to studying prime ideals in $\mathcal{O}G$ that do not contain $p$.\\ 

\noindent The methods we use are inspired by those used in \cite{nilpotent} to prove that faithful prime ideals in the mod-$p$ Iwasawa algebra $kG$ are standard, namely the study of \emph{Mahler expansions of $G$-automorphisms}. Unfortunately, these methods do not work in characteristic 0, as demonstrated in \cite[Section 3.3]{primitive}, and the best result they can be used to obtain is a weak control theorem for faithful \emph{primitive} ideals (\cite[Theorem 1.2]{primitive}).

The methods we will employ in this section involve considering \emph{logarithms} of $G$-automorphisms, which in many ways is the characteristic 0 version of their Mahler expansions. And while these methods are not yet sufficient to prove standardness for all prime ideals, we can employ them together with techniques from \cite[Section 7]{nilpotent} to adapt the argument used in \cite[Section 3.4]{primitive} and ultimately reprove the weak control theorem \cite[Theorem 1.2]{primitive} for all faithful prime ideals, and not just primitives. This will culminate in the proof of Theorem \ref{C}.

\subsection{Bounded Ring Automorphisms}

Let $R$ be a ring carrying a complete Zariskian filtration $w$. Recall from \cite{nilpotent} that a function $f:R\to R$ is \emph{bounded} if $\inf\{w(f(r))-w(r):r\in R\}>-\infty$, in which case we define the \emph{degree} of $f$ to be the number $\deg_w(f):=\inf\{w(f(r))-w(r):r\in R\}$. 

If we set $B(R)$ as the space of bounded, additive maps $f:R\to R$, then $B(R)$ is a ring with pointwise addition and composition as multiplication, and $\deg_w$ defines a complete separated filtration on $B(R)$.

\begin{lemma}\label{quotient}

If $f:R\to R$ is an additive map such that $\deg_w(f)>0$, and $I$ is a two-sided ideal of $R$ such that $f(I)\subseteq I$. Then if $\overline{w}$ is the quotient filtration on $R/I$, and $\overline{f}:R/I\to R/I$ is the map induced from $f$, then $\deg_{\overline{w}}(\overline{f})>0$.

\end{lemma}

\begin{proof}

Let $\mu:=\deg_w(f)>0$. Then given $r\in R$, $w(f(r))-w(r)\geq\mu$, we want to prove that $\overline{w}(\overline{f}(r+I))-\overline{w}(r+I)>0$:\\

\noindent By definition, $\overline{w}(r+I)=\sup\{w(r+y):y\in I\}$, so let us suppose for contradiction that $\overline{w}(r+I)\geq\overline{w}(\overline{f}(r+I))$, and hence there exists $y\in I$ such that $w(r+y)\geq w(f(r)+u)$ for all $u\in I$. In particular, since $f(I)\subseteq I$, $w(r+y)\geq w(f(r)+f(y))=w(f(r+y))\geq w(r+y)+\mu>w(r+y)$ -- contradiction.

Therefore $\overline{w}(\overline{f}(r+I))>\overline{w}(r+I)$ for all $r\in R$, so since $\overline{w}$ is integer valued, it follows that $\deg_{\overline{w}}(\overline{f})\geq 1>0$.\end{proof}

\noindent Now, suppose that $R$ is a $\mathbb{Z}_p$-algebra, with $p\neq 0$ and $w(p)>0$. The following lemma will be useful to us several times in this section:

\begin{lemma}\label{terms2}

Given $m\in\mathbb{N}$, $a,b\in R$ such that $w(a)=0$, $w(b)\geq 1$, and $a$ and $b$ commute:

\begin{itemize}

\item $v_p\binom{p^m}{k}=m-v_p(k)$ for all $0<k<p^m$.

\item $\min\{m-v_p(k)+(p^m-k)w(b):0<k<p^m\}\rightarrow\infty$ as $m\rightarrow\infty$

\item $w((a+b)^{p^m}-a^{p^m})\geq\min\{p^mw(b),m-v_p(k)+(p^m-k)w(b):0<k<p^m\}$, and thus $(a+b)^{p^m}-a^{p^m}\rightarrow 0$ as $m\rightarrow\infty$.

\end{itemize}

\end{lemma}

\begin{proof}

Firstly, if $k=a_0+a_1p+\cdots+a_tp^t$ for some $0\leq a_i<p$, we define $s(k):=a_0+a_1+\cdots+a_t$. Then using \cite[\rom{3} 1.1.2.5]{Lazard} we see that $v_p(k!)=\frac{k-s(k)}{p-1}$. Therefore, $v_p\binom{p^m}{k}=v_p\left(\frac{p^m!}{k!(p^m-k)!}\right)=\frac{p^m-s(p^m)-k+s(k)-(p^m-k)+s(p^m-k)}{p-1}=\frac{s(k)+s(p^m-k)+1}{p-1}$.\\

\noindent But since $k<p^m$, we may assume that $t=m-1$, i.e. $k=a_0+a_1p+\cdots+a_{m-1}p^{m-1}$. And since $k\neq 0$, let $i$ be maximal such that $a_{m-i}\neq 0$, so $k=a_{m-i}p^{m-i}+\cdots+a_{m-1}p^{m-1}$ and hence $v_p(k)=m-i$.\\

\noindent Now, $p^m=(p-1)p^{m-i}+(p-1)p^{m-i+1}+\cdots+(p-1)p^{m-1}+p^{m-i}$, and thus $p^m-k=(p-a_{m-i})p^{m-i}+(p-a_{m-i+1}-1)p^{m-i+1}+\cdots+(p-a_{m-1}-1)p^{m-1}$ and we deduce that $s(p^m-k)=ip-s(k)-(i-1)=i(p-1)-s(k)+1$.\\

\noindent Therefore, $v_p\binom{p^m}{k}=\frac{s(k)+s(p^m-k)-1}{p-1}=\frac{i(p-1)}{p-1}=i=m-v_p(k)$ as required.\\

\noindent To prove the second statement, we just need to prove that for any $k$, $m-v_p(k)+(p^m-k)w(b)\rightarrow\infty$ as $m\rightarrow\infty$:\\

\noindent If $v_p(k)\leq\frac{m}{2}$ then $m-v_p(k)+(p^m-k)w(b)\geq\frac{m}{2}\rightarrow\infty$.\\

\noindent If $v_p(k)>\frac{m}{2}$ then $k=p^{v_p(k)}y$ with $v_p(y)=0$, so $m-v_p(k)+(p^m-k)w(b)\geq p^{v_p(k)}(p^{m-v_p(k)}-y)w(b)$

\noindent $\geq p^{\frac{m}{2}}\rightarrow\infty$ as required.\\

\noindent Finally, $(a+b)^{p^m}-a^{p^m}=\underset{0\leq k<p^m}{\sum}{\binom{p^m}{k}a^kb^{p^m-k}}$, so $w((a+b)^{p^m}-a^{p^m})\geq\min\{p^mw(b),w(\binom{p^m}{k}a^kb^{p^m-k}):0<k<p^m\}$, and $w(\binom{p^m}{k}a^kb^{p^m-k})\geq v_p\binom{p^m}{k}w(p)+kw(a)+(p^m-k)w(b)\geq m-v_p(k)+(p^m-k)w(b)$ as required.\end{proof}

\noindent Now, fix a ring automorphism $\varphi:R\to R$ such that $\deg_w(\varphi-1)>0$.

\begin{lemma}\label{aut-converge}

The sequence $\varphi^n-1$ converges to $0$ in $B(R)$ as $v_p(n)\rightarrow\infty$.

\end{lemma}

\begin{proof}

We just need to prove that $\deg_w(\varphi^n-1)\rightarrow\infty$ as $v_p(n)\rightarrow\infty$. Since $\deg_w(\varphi^n-1)>0$ for all $n$, it suffices to prove that $\deg_w(\varphi^{p^m}-1)\rightarrow\infty$ as $m\rightarrow\infty$.\\

\noindent Now, $\varphi^{p^m}-1=((\varphi-1)+1)^{p^m}-1=\underset{0\leq k< p^m}{\sum}{\binom{p^m}{k}(\varphi-1)^{p^m-k}}$. But since $\deg_w$ defines a ring filtration on $B(R)$ and $\deg_w(\varphi-1)\geq 1$, it follows from Lemma \ref{terms2} that if $k>0$ then\\ 

\noindent $\deg_w(\binom{p^m}{k}(\varphi-1)^{p^m-k})\geq w\binom{p^m}{k}+(p^m-k)\deg_w(\varphi-1)\geq m-v_p(k)+(p^m-k)\deg_w(\varphi-1)\rightarrow\infty$ as $m\rightarrow\infty$, and clearly if $k=0$ then $\deg_w(\binom{p^m}{k}(\varphi-1)^{p^m-k})\geq p^m\deg_w(\varphi-1)\rightarrow\infty$ as required.\end{proof}

\vspace{0.1in}

\noindent Now, let us suppose that $R$ is a prime, Noetherian ring, gr$_w$ $R$ is commutative, and that the positively graded ideal $($gr$_w$ $R)_{\geq 0}$ is not nilpotent. Then the simple, artinian ring $Q(R)$ carries a non-commutative valuation $v$, which we can describe using Construction \ref{non-comm-val}. Clearly any automorphism $\varphi$ of $R$ extends to $Q(R)$.\\

\noindent The following theorem allows us to pass from $(R,w)$ to $(Q(R),v)$ without difficulty. The proof is similar to that of an analogous result in a characteristic $p$ setting, namely the proof of \cite[Proposition 3.5]{APB}.

\begin{theorem}\label{pos-degree}

If $\varphi\in$ Aut$(R)$ and $\deg_w(\varphi-1)>0$, then there exists $n\in\mathbb{N}$ with $\deg_v(\varphi^n-1)>0$.

\end{theorem}

\begin{proof}

Using Construction \ref{non-comm-val}, we see that we have a sequence of filtrations $w',v_{z,U},v_{\underline{z},V},v_{\underline{z},\mathcal{B}},v$ on $Q(R)$, and our strategy is to prove that $\varphi^n-1$ has positive degree for some $n$ with respect to each of these in turn. Let us first consider $w'$.\\

\noindent Since $\deg_w(\varphi-1)>0$, i.e. $w(\varphi(r)-r)>w(r)$ for all $r\in R$, it follows that the induced graded automorphism $\overline{\varphi}:$ gr$_w$ $R\to$ gr$_w$ $R,r+F_nR\mapsto\varphi(r)+F_nR$ is just the identity. Therefore, since gr$_{w'}$ $Q(R)=($gr $R)_{\mathfrak{q}}$, it follows that the induced morphism $\overline{\varphi}:$ gr$_{w'}$ $Q(R)\to$ gr$_{w'}$ $Q(R)$ is also the identity, and hence $\deg_{w'}(\varphi-1)\geq 1$.\\

\noindent So, using Lemma \ref{aut-converge}, we can choose $n_1\in\mathbb{N}$ such that $\deg_{w'}(\varphi^{n_1}-1)\geq w'(z)$. But since $z^nU=F_{nw'(z)}Q(R)$, it follows that if $r\in z^nU$ then $w'((\varphi^{n_1}-1)(r))\geq w'(r)+w'(z)\geq nw'(z)+w'(z)=(n+1)w'(z)$, and hence $(\varphi^{n_1}-1)(z^nU)\subseteq F_{(n+1)w'(z)}Q(R)=z^{n+1}U$, and hence $\deg_{v_{z,U}}(\varphi^{n_1}-1)\geq 1$.\\

\noindent Now, recall that $\widehat{Q}$ is a simple quotient of the completion $Q'$ of $Q(R)$ with respect to $v_{z,U}$, i.e. a quotient of $Q'$ by a maximal ideal $\mathfrak{m}$. But since $Q'$ is artinian, all maximal ideals are minimal prime ideals, and hence there are only finitely many of them. Since $\varphi^{n_1}$ is continuous, $\varphi^{n_1}(\mathfrak{m})$ is also a minimal prime ideal of $Q'$, i.e. $\varphi^{n_1}$ permutes the set of minimal prime ideals of $Q'$. Since this set is finite, all permutations have finite order, so there exists $m$ such that $\varphi^{mn_1}(\mathfrak{m})=\mathfrak{m}$. 

Therefore, there exists $n_2$ such that $\varphi^{n_2}$ induces an automorphism $\overline{\varphi^{n_2}}:\widehat{Q}\to\widehat{Q}$, and since $Q(R)$ is simple, the composition $Q(R)\xhookrightarrow{} Q'\twoheadrightarrow\widehat{Q}$ must be injective. Therefore, we can think of $\overline{\varphi^{n_2}}$ as a extension of $\varphi^{n_2}$ to $\widehat{Q}$, so sometimes we may just call it $\varphi^{n_2}$ for convenience.\\

\noindent Now, if $r\in\overline{z}^nV$, then $r$ is the image of $z^ns$ in $\widehat{Q}$, for some $s\in U$. So since $\deg_{v_{z,U}}(\varphi^{n_2}-1)\geq 1$, it follows that $(\varphi^{n_2}-1)(z^ns)\in z^{n+1}U$, and hence $(\overline{\varphi^{n_2}}-1)(r)\in\overline{z}^{n+1}V$. Therefore $\deg_{v_{\overline{z},V}}(\overline{\varphi^{n_2}}-1)\geq 1$.\\

\noindent Now, $\mathcal{B}$ is a maximal order in $\widehat{Q}$, equivalent to $V$, and $\mathcal{B}\subseteq\overline{z}^{-r}V$. So let $I:=\{x\in V:Bx\subseteq V\}$, then $I$ is a two-sided ideal of $V$ with $z^rV\subseteq I$. Since $\deg_{v_{\overline{z},V}}(\varphi^{n_2}-1)>0$, it follows from Lemma \ref{aut-converge} that we can choose $m\in\mathbb{N}$ such that $(\varphi^{mn_2}-1)(V)\subseteq z^rV$, and hence there exists $n_3\in\mathbb{N}$ such that $(\varphi^{n_3}-1)(V)\subseteq I$.

In particular, $\varphi^{n_3}(I)\subseteq I$, and it follows from Noetherianity of $V$ that $\varphi^{n_3}(I)=I$, and hence $I=\varphi^{-n_3}(I)$.\\

\noindent Therefore, given $x\in\varphi^{n_3}(\mathcal{B})$, $x=\varphi^{n_3}(b)$ for some $b\in\mathcal{B}$, and given $c\in I$, $xc=\varphi^{n_3}(b)c=\varphi^{n_3}(b\varphi^{-n_3}(c))$, and since $\varphi^{-n_3}(c)\in I$, it follows that $b\varphi^{-n_3}(c)\in V$, and hence $xc\in V$.\\

\noindent So setting $\mathcal{O}(I)=\{b\in B:bI\subseteq V\}$, it follows that $\mathcal{O}(I)$ contains $\mathcal{B}$ and $\varphi^{n_3}(\mathcal{B})$. But $\mathcal{O}(I)$ is an order in $\widehat{Q}$, equivalent to $V$, by \cite[Lemma 5.1.12]{McConnell}, so since $\mathcal{B}$ and $\varphi^{n_3}(\mathcal{B})$ are maximal orders, this means that $\mathcal{O}(I)=\mathcal{B}=\varphi^{n_3}(\mathcal{B})$, and hence $\varphi^{n_3}$ is an automorphism of $\mathcal{B}$.\\

\noindent Also, we can choose $n_3$ such that $\deg_{v_{\overline{z},V}}(\varphi^{n_3}-1)\geq r+1$, and hence $(\varphi^{n_3}-1)(\overline{z}^n\mathcal{B})\subseteq (\varphi^{n_3}-1)(z^{n-r}V)\subseteq z^{n+1}v\subseteq z^{n+1}\mathcal{B}$. Therefore $\deg_{\overline{z},\mathcal{B}}(\varphi^{n_3}-1)\geq 1$.\\

\noindent Finally, $\mathcal{B}=M_n(D)$ for some non-commutative DVR $D$, so let $\nu$ be a uniformiser in $D$, and it follows that all two-sided ideals of $\mathcal{B}$ have the form $\nu^n\mathcal{B}$. Since $v$ is the $\nu$-adic filtration on $\widehat{Q}$, and $v$ is topologically equivalent to $v_{\overline{z},\mathcal{B}}$, we know that there exists $k\in\mathbb{N}$ such that $z^r\mathcal{B}\subseteq\nu^2\mathcal{B}$. By Lemma \ref{aut-converge}, we can choose $n\geq n_3$ such that $\deg_{v_{\overline{z},\mathcal{B}}}(\varphi^n-1)\geq k$, i.e. $(\varphi^n-1)(\mathcal{B})\subseteq\overline{z}^k\mathcal{B}\subseteq\mu^2\mathcal{B}$.\\

\noindent So, if we assume for induction that $(\varphi^n-1)(\nu^i\mathcal{B})\subseteq\nu^{i+1}\mathcal{B}$ for all $i<m$, then $(\varphi^n-1)(\nu^m)=\varphi^n(\nu^m)-\nu^m=(\varphi^n-1)(\nu^{m-1})\varphi^n(\nu)+\nu^{m-1}(\varphi^n-1)(\nu)$. But $(\varphi^n-1)(\nu)\in\nu^2\mathcal{B}$, $(\varphi^n-1)(\nu^{m-1})\in\nu^m\mathcal{B}$ and $\varphi^n(\nu)\in\nu\mathcal{B}$, therefore $(\varphi^n-1)(\nu^m)\in\nu^{m+1}\mathcal{B}$. 

It follows that $(\varphi^n-1)(\nu^m\mathcal{B})\subseteq\nu^{m+1}\mathcal{B}$ for all $m$, and hence $\deg_v(\varphi^n-1)\geq 1$ as required.\end{proof}

\subsection{Bounded Group Automorphisms}

Now, let $(G,\omega)$ be a complete, $p$-valued group of finite rank. We may assume that $\omega$ is an \emph{abelian $p$-valuation} as defined in Definition \ref{abelian-val}, and we let $w$ be the corresponding Lazard filtration on $\mathcal{O}G$ -- a complete Zariskian filtration.\\

\noindent Fix an ordered basis $\underline{g}=\{g_1,\cdots,g_d\}$ for $(G,\omega)$, so that $\mathcal{O}G$ is isomorphic to the space of power series $\mathcal{O}[[b_1,\cdots,b_d]]$ as an $\mathcal{O}$-module, where $b_i=g_i-1$. Recall that for any $g\in G$, $w(g-1)\geq en\omega(g)$, with equality if $g=g_i$ for some $i$.\\

\noindent Now, recall the following definition (\cite[Definition 4.5]{nilpotent}):

\begin{definition}\label{bound}

An automorphism $\varphi\in$ Aut$(G)$ is \emph{bounded} if $\inf\{\omega(\varphi(g)g^{-1})-\omega(g):g\in G\}>\frac{1}{p-1}$, and we define the \emph{degree} of $\varphi$ to be the number $\deg_{\omega}(\varphi)=\inf\{\omega(\varphi(g)g^{-1})-\omega(g):g\in G\}$. 

Let Aut$^{\omega}(G)$ be the group of bounded automorphisms of $G$.

\end{definition}

\begin{lemma}\label{G-OG}

If $\varphi\in$ Aut$^{\omega}(G)$ then $\varphi$ extends to a continuous, $\mathcal{O}$-linear automorphism of $\mathcal{O}G$ such that $\deg_w(\varphi-1)>0$.

\end{lemma}

\begin{proof}

Clearly $\varphi$ extends to an $\mathcal{O}$-linear automorphism of $\mathcal{O}G$, so we need only prove that $\deg_w(\varphi-1)>0$, and it will follow that $\varphi$ is continuous and extends to $\mathcal{O}G$.\\

\noindent Firstly, for each $i=1,\cdots,d$, $(\varphi-1)(g_i-1)=\varphi(g_i)-g_i=(\varphi(g_i)g_i^{-1}-1)g_i$. But since $\varphi\in Aut^{\omega}(G)$, we know that $\omega(\varphi(g_i)g_i^{-1})>\omega(g)+\frac{1}{p-1}$, so $w(\varphi(g_i)g_i^{-1}-1)\geq en\omega(\varphi(g_i)g_i^{-1})\geq en\omega(g_i)=w(g_i-1)$. Hence $w((\varphi-1)(g_i-1))-w(g_i-1)>0$.\\

\noindent Now, given $a,b\in\mathcal{O}[G]$: 

\begin{center}
$(\varphi-1)(ab)=\varphi(a)\varphi(b)-ab=(\varphi(a)-a)\varphi(b)+a(\varphi(b)-b)=(\varphi-1)(a)\varphi(b)+a(\varphi-1)(b)$. 
\end{center}

\noindent If we assume that $w((\varphi-1)(a))-w(a)>0$ and $w((\varphi-1)(b))-w(b)>0$, then it follows that $w(\varphi(b))=w(b)$ and $w((\varphi(a)-a)\varphi(b)+a(\varphi(b)-b))>\min\{w(a)+w(\varphi(b),w(a)+w(b)\}=w(a)+w(b)=w(ab)$, therefore $w((\varphi-1)(ab))-w(ab)>0$.\\

\noindent So since $w((\varphi-1)(g_i-1))-w(g_i-1)>0$ for each $i=1,\cdots,d$, it follows that for every $\alpha\in\mathbb{N}^d$, $\lambda\in\mathcal{O}$, $w(\lambda(\varphi-1)((g_1-1)^{\alpha_1}\cdots (g_d-1)^{\alpha_d}))-w(\lambda (g_1-1)^{\alpha_1}\cdots (g_d-1)^{\alpha_d})>0$. In particular, since $\mathcal{O}[G]$ is generated as an additive group by the monomials $\lambda b_1^{\alpha_1}\cdots b_d^{\alpha_d}$, it follows that $w(\varphi(s))=w(s)$ for all $s\in\mathcal{O}[G]$.\\

\noindent Now, given $r\in\mathcal{O}[G]$ with $w(r)=t$, it follows from the definition of $w$ that $r=\underset{\alpha\in A}{\sum}{\lambda_{\alpha}b_1^{\alpha_1}\cdots b_d^{\alpha_d}}+s$, where $A:=\{\alpha\in\mathbb{N}^d:\underset{1\leq i\leq d}{\sum}{\alpha_ien\omega(g_i)}=t\}$ and $w(s)>t$. Since $w(\varphi(s))=w(s)$, it is clear that $w((\varphi-1)(s))\geq w(s)>t$, so to prove that $w((\varphi-1)(r))>t$, it remains to show that $w((\varphi-1)(\underset{\alpha\in A}{\sum}{\lambda_{\alpha}b_1^{\alpha_1}\cdots b_d^{\alpha_d}}))>t$.\\

\noindent But $w(\lambda_{\alpha}b_1^{\alpha_1}\cdots b_d^{\alpha_d})=t$ for all $\alpha\in A$, so we have seen that $w((\varphi-1)(\lambda_{\alpha}b_1^{\alpha_1}\cdots b_d^{\alpha_d}))>t$ for each $\alpha$, and it follows immediately that $w((\varphi-1)(\underset{\alpha\in A}{\sum}{\lambda_{\alpha}b_1^{\alpha_1}\cdots b_d^{\alpha_d}}))>t$ as required.\end{proof}

\noindent Now, fix a prime ideal $P$ of $\mathcal{O}G$ such that $p\notin P$, and fix an automorphism $\varphi\in Aut^{\omega}(G)$ such that $P$ is invariant under the extension of $\varphi$ to $\mathcal{O}G$, i.e. $\varphi(P)=P$. Hence $\varphi$ induces an automorphism $\overline{\varphi}$ of $\frac{\mathcal{O}G}{P}$.

\begin{theorem}\label{aut-completion}

There exists a non-commutative valuation $v$ on $Q(\mathcal{O}G/P)$ such that the natural map $\tau:(\mathcal{O}G,w)\to (Q(\mathcal{O}G/P),v)$ is continuous, and there exists $n\in\mathbb{N}$ such that  $\deg_v(\overline{\varphi}^n-1)\geq v(p)$.

\end{theorem}

\begin{proof}

Let $\overline{w}$ be the quotient filtration on $\frac{\mathcal{O}G}{P}$, so that gr$_{\overline{w}}$ $\frac{\mathcal{O}G}{P}\cong\frac{\text{gr }_w\text{ }\mathcal{O}G}{\text{gr}_w\text{ }P}$. So since $w$ is Zariskian, it follows that gr$_{\overline{w}}$ $\frac{\mathcal{O}G}{P}$ is commutative and Noetherian, and since $\frac{\mathcal{O}G}{P}$ is complete with respect to $\overline{w}$, it follows from \cite[Theorem 2.1.2]{LVO} that $\overline{w}$ is Zariskian.\\

\noindent Moreover, since $p\notin P$ and $w(p)>0$, it follows that $\overline{w}(p^n)=w(p^n)$ for all $n\in\mathbb{N}$, because if $\overline{w}(p^n)>w(p^n)$ then $p^n+r\in P$ for some $r\in\mathcal{O}G$ with $w(r)>w(p^n)$, and hence $1+p^{-n}r\in P\otimes_{\mathcal{O}}K$ with $w(p^{-n}r)>0$, but $1+p^{-n}r$ is a unit in $KG$ -- a contradiction. Therefore, it follows that gr$(p)$ lies in the positively graded piece of the associated graded gr $\frac{\mathcal{O}G}{P}$ and is non-nilpotent. Hence $($gr$_{\overline{w}}$ $\frac{\mathcal{O}G}{P})_{\geq 0}$ is non-nilpotent.\\

\noindent Therefore, using Construction \ref{non-comm-val}, we can define a non-commutative valuation $v$ on $Q(\mathcal{O}G/P)$ such that the inclusion $(\mathcal{O}G/P,\overline{w})\to (Q(\mathcal{O}G/P),v)$ is continuous. Since the surjection $(\mathcal{O}G,w)\to (\mathcal{O}G/P,\overline{w})$ is continuous, it follows that the composition $\tau$ is also continuous.\\

\noindent Now, since $\deg_w(\varphi-1)>0$ by Lemma \ref{G-OG}, and $(\varphi-1)(P)\subseteq P$, it follows from Lemma \ref{quotient} that $\deg_{\overline{w}}(\varphi-1)>0$. Using Theorem \ref{pos-degree}, it follows that there exists $n\in\mathbb{N}$ such that $\deg_v(\varphi^n-1)>0$ as required.\end{proof}

\subsection{The Logarithm}

Given a complete $p$-valued group $(G,\omega)$ of finite rank, and a faithful prime ideal $P$ of $\mathcal{O}G$ such that $p\notin P$, we now want to take steps towards proving a control theorem for $P$. Again, assume that there is an automorphism $\varphi\in$ Aut$^{\omega}(G)$ with $\varphi\neq 1$ such that $\varphi(P)=P$, and let $\overline{\varphi}$ be the automorphism of $\mathcal{O}G/P$ induced from $\varphi$.\\

\noindent We will now assume further that there is a closed, central subgroup $A$ of $G$ such that:

\begin{itemize}

\item $\varphi(g)g^{-1}\in A$ for all $g\in G$.

\item For all $a\in A$, $\varphi(a)=a$.

\end{itemize} 

\noindent It is straightforward to show that these properties are satisfied for $\varphi^n$ for all $n\in\mathbb{N}$.\\

\noindent Now, using Theorem \ref{aut-completion}, we fix a non-commutative valuation $v$ on $Q=Q(\mathcal{O}G/P)$ such that the natural map $\tau:(\mathcal{O}G,w)\to (Q,v)$ is continuous, and a natural number $n\in\mathbb{N}$ such that $\deg_v(\overline{\varphi}^n-1)\geq v(p)$. After replacing $\varphi$ by $\varphi^n$ if necessary, we may assume that $n=1$. 

Let $\widehat{Q}$ be the completion of $Q$ with respect to $v$, and clearly $\overline{\varphi}$ extends continuously to $\widehat{Q}$.\\

\noindent Also, since $A$ is central in $G$, $P\cap\mathcal{O}A$ is a prime ideal of $\mathcal{O}A$. Therefore, the field of fractions of $\frac{\mathcal{O}A}{P\cap\mathcal{O}A}$ is contained in $\widehat{Q}$. So, let $F$ be the closure of this field of fractions in $\widehat{Q}$, then $F$ is a central subfield of $\widehat{Q}$ carrying a complete valuation $v_F=v|_F$.

\begin{definition}

Define the \emph{logarithm} of $\overline{\varphi}$ to be the derivation of $\widehat{Q}$ defined by the logarithm series at $\overline{\varphi}$. Specifically:

\begin{equation}
\log(\overline{\varphi})=\underset{k\geq 1}{\sum}{\frac{(-1)^{k+1}}{k}(\overline{\varphi}-1)^k}
\end{equation}

\end{definition}

\noindent\textbf{Note:} 1. This definition never makes sense if $p\in P$, because in this case if $p\mid k$ then $k=0$ in $\widehat{Q}$.\\

\noindent 2. This definition makes sense when $p\notin P$ because $\deg_v(\overline{\varphi}-1)\geq v(p)$ so $\deg_v(\overline{\varphi}-1)^k\geq kv(p)$ for all $k$ and $\deg_v(\frac{(-1)^{k+1}}{k}(\overline{\varphi}-1)^k)\geq kv(p)-v(k)=kv(p)-v_p(k)v(p)\rightarrow 0$ as $k\rightarrow\infty$. So since $B(\widehat{Q})$ is complete with respect to $\deg_v$ by \cite[Lemma 2.4]{nilpotent}, it follows that the logarithm series must converge to an element of $B(\widehat{Q})$, and since $\overline{\varphi}$ is an automorphism, this must be a derivation.\\

\noindent 3. See the proof of \cite[Theorem 4]{derivation} for details of why $\log(\overline{\varphi})$ is a derivation of $\widehat{Q}$.

\begin{proposition}\label{log-properties}

Fix $g\in G$, then:\\

$i$. $\deg(\log(\overline{\varphi}))\geq 1$.

$ii$. The series $\log(\varphi(g)g^{-1}):=\underset{k\geq 1}{\sum}{\frac{(-1)^{k+1}}{k}\tau(\varphi(g)g^{-1}-1)^k}$ converges in $F$.

$iii$. $\log(\overline{\varphi})(g)=\log(\varphi(g)g^{-1})\tau(g)$.

$iv$. $\log(\overline{\varphi})(F)=0$

$v$. $\log(\overline{\varphi})$ is a continuous, $F$-linear derivation.

\end{proposition}

\begin{proof}

$i$. Since $\deg_v(\overline{\varphi})\geq v(p)$, it follows that $\deg(\frac{(-1)^{k+1}}{k}(\overline{\varphi}-1)^k)\geq kv(p)-v(k)=(k-v_p(k))v(p)\geq 1$, and hence $\deg_v(\log(\overline{\varphi}))\geq 1$.\\

\noindent $ii$ -- $iii$. For any $g\in G$, $(\overline{\varphi}-1)(g)=\tau(\varphi(g)-g)=\tau(\varphi(g)g^{-1}-1)\tau(g)$, and if we suppose for induction that $(\overline{\varphi}-1)^k(g)=\tau(\varphi(g)g^{-1}-1)^k\tau(g)$, then $(\overline{\varphi}-1)^{k+1}(g)=(\overline{\varphi}-1)(\tau(\varphi(g)g^{-1}-1)^k\tau(g))=\overline{\varphi}(\tau(\varphi(g)g^{-1}-1)^k)\overline{\varphi}(\tau(g))-\tau(\varphi(g)g^{-1}-1)^k\tau(g)$.

But since $\varphi(g)g^{-1}\in A$, it follows from our assumption that $\overline{\varphi}(\tau(\varphi(g)g^{-1}-1)^k)=\tau(\varphi(g)g^{-1}-1)^k$, and hence $(\overline{\varphi}-1)^{k+1}(g)=\tau(\varphi(g)g^{-1}-1)^k\overline{\varphi}(\tau(g))-\tau(\varphi(g)g^{-1}-1)^k\tau(g)=\tau(\varphi(g)g^{-1}-1)^{k+1}\tau(g)$.\\

\noindent Therefore, $\log(\overline{\varphi})(g)=\underset{k\geq 1}{\sum}{\frac{(-1)^{k+1}}{k}\tau(\varphi(g)g^{-1}-1)^k\tau(g)}$, so multiplying on the right by $\tau(g)^{-1}$ gives that $\log(\varphi(g)g^{-1})=\underset{k\geq 1}{\sum}{\frac{(-1)^{k+1}}{k}\tau(\varphi(g)g^{-1}-1)^k}$ converges to $\log(\overline{\varphi})(g)\tau(g)^{-1}$, and since $\tau(\varphi(g)g^{-1}-1)\in\tau(\mathcal{O}A)\subseteq F$, it follows that $\log(\varphi(g)g^{-1})\in F$.\\

\noindent $iv$. Again, since $\varphi(a)=a$ for all $a\in A$, it follows that $(\overline{\varphi}-1)(\tau(a))=0$, and hence $\log(\overline{\varphi})(a)=0$. So since $\log(\varphi)$ is $\mathcal{O}$-linear and continuous, it follows that $\log(\overline{\varphi})(s)=0$ for all $s\in\tau(\mathcal{O}A)=\mathcal{O}A/\mathcal{O}A\cap P$.

Moreover, since $F$ is the completion of the field of fractions of $\tau(\mathcal{O}A)$, and $\log(\overline{\varphi})$ is a derivation, it follows that $\log(\overline{\varphi})(s)=0$ for all $s\in F$ as required.\\

\noindent $v$. We know that $\log(\overline{\varphi})$ is a continuous derivation, and since $\log(\overline{\varphi})(F)=0$ it follows that it is $F$-linear.\end{proof}

\noindent\textbf{Remark:} Without our assumptions that $\varphi(g)g^{-1}\in A$ for all $g$, and $\varphi$ is trivial when restricted to $A$, this proposition fails.\\

\noindent Fix $\lambda:=\inf\{v(\log(\overline{\varphi})(\varphi(g)g^{-1})):g\in G\}$, and let $U:=\{g\in G:v(\log(\varphi(g)g^{-1}))>\lambda\}$. The following lemma depends on the assumption that $P$ is faithful:

\begin{lemma}\label{faithful}

$1\leq \lambda<\infty$, and $U$ is a proper, open subgroup of $G$ containing $G^p$ and $(G,G)$.

\end{lemma}

\begin{proof}

For convenience, set $\psi(g):=\varphi(g)g^{-1}$.  Then since $\varphi(g)g^{-1}\in Z(G)$ for all $g\in G$, it follows that $\psi:G\to A$ is a group homomorphism.\\

Using Proposition \ref{log-properties}, we see that $\deg(\overline{\varphi})\geq 1$, and that $\log(\overline{\varphi})(g)=\log(\psi(g))g$ for all $g\in G$. So since $v(g)=0$, it follows that $v(\log(\psi(g)))=v(\log(\overline{\varphi})(g))\geq\deg(\overline{\varphi})\geq 1$ for all $g\in G$, and hence $\lambda\geq 1$.\\

\noindent If $\lambda=\infty$ then $v(\log(\psi(g)))=\infty$ and hence $\log(\psi(g))=0$ for all $g\in G$. But the function on $\widehat{Q}$ defined by the logarithm series is injective by \cite[Corollary 6.25($ii$)]{DDMS}, and hence $\tau(\psi(g)-1)=0$ for all $g\in G$, i.e. $\psi(g)-1\in P$. But $P$ is faithful, so this means that $\psi(g)=\varphi(g)g^{-1}=1$ for all $g\in G$, and hence $\varphi$ is trivial -- contradicting our assumption.\\

\noindent Therefore $1\leq \lambda<\infty$, and clearly $\lambda$ is an integer, so there exists $g\in G$ such that $v(\log(\psi(g)))=\lambda$, and hence $U\subsetneq G$.\\

\noindent Furthermore, given $g\in G$, $v(\log(\psi(g^p)))=v(\log(\psi(g)^p))=v(p\log(\psi(g))=v(\log(\psi(g)))+v(p)\geq\lambda+v(p)>\lambda$, and hence $G^p\subseteq U$.\\

\noindent Also, if $g,h\in U$ then 

\begin{center}
$v(\log(\psi(gh)))=v(\log(\psi(g)\psi(h)))=v(\log(\psi(g))+\log(\psi(h)))\geq\min\{v(\log(\psi(g))),v(\log(\psi(h)))\}>\lambda$,
\end{center} 

\noindent so since $\frac{G}{G^p}$ is a finite group and $\frac{U}{G^p}$ is closed under the group operation, it follows that $U$ is an open subgroup of $G$.\\

\noindent Finally, if $g,h\in G$, $v(\log(\psi(g,h)))=v(\log((\psi(g),\psi(h))))=v(\log(1))=\infty>\lambda$, so $(G,G)\subseteq U$ as required.\end{proof}

\noindent So, using \cite[Lemma 4.2]{nilpotent}, choose an ordered basis $\{g_1,\cdots,g_d\}$ for $(G,\omega)$ such that $\{g_1^p,\cdots,g_r^p,g_{r+1},\cdots,g_d\}$ is an ordered basis for $U$ for some $1\leq r\leq d$. Thus $v(\log(\varphi(g_i)g_i^{-1}))=\lambda$ for $i=1,\cdots,r$ and $v(\log(\varphi(g_i)g_i^{-1}))>\lambda$ for all $i>r$.\\

\noindent Now, set $a:=\log(\varphi(g_1)g_1^{-1})$. Then $a\in F$ by Proposition \ref{log-properties}, and $v(a)=\lambda<\infty$ so $a\neq 0$, thus $a$ is a unit in $F$. So for each $i=1,\cdots,d$, set $z_i:=a^{-1}\log(\varphi(g_i)g_i^{-1})\in F$.

Since $F$ is central, it follows from the definition of a non-commutative valuation that $v(z_i)=v(\log(\varphi(g_i)g_i^{-1}))-v(a)$ for each $i$, so $v(z_1)=\cdots=v(z_r)=0$, and $v(z_i)>0$ if $i>r$.\\

\noindent From now on, set $\delta:=a^{-1}\log(\overline{\varphi}):\widehat{Q}\to\widehat{Q}$, which is an $F$-linear derivation of $\widehat{Q}$, and using Proposition \ref{log-properties} we see that for all $g\in G$, $\delta(g)=a^{-1}\log(\overline{\varphi})(g)=a^{-1}\log(\varphi(g)g^{-1})g$. But if $g=\underline{g}^{\alpha}:=g_1^{\alpha_1}\cdots g_d^{\alpha_d}$ for some $\alpha\in\mathbb{Z}_p^d$ then 

\begin{center}
$\log(\varphi(g)g^{-1})=\log((\varphi(g_1)g_1^{-1})^{\alpha_1}\cdots(\varphi(g_d)g_d^{-1})^{\alpha_d}))=\alpha_1\log(\varphi(g_1)g_1^{-1})+\cdots+\alpha_d\log(\varphi(g_d)g_d^{-1})$.
\end{center}

\noindent Therefore, $\delta(\underline{g}^{\alpha})=(\alpha_1z_1+\cdots+\alpha_dz_d)\underline{g}^{\alpha}$ for all $\alpha\in\mathbb{Z}_p^d$.\\

\noindent Furthermore, since $\delta$ is $F$-linear, it follows that $\delta^n(\underline{g}^{\alpha})=(\alpha_1z_1+\cdots\alpha_dz_d)^n\underline{g}^{\alpha}$ for all $n\in\mathbb{N}$.

\subsection{A Convergence argument}

We will now show how we can study convergence of $\delta^{p^m}$ to prove a control theorem for $P$.\\

\noindent\textbf{Notation:} For any $\alpha\in\mathbb{Z}_p$, denote by $\alpha'$ the unique integer in $\{0,\cdots,p-1\}$ such that $\alpha\equiv\alpha'$ $($mod $p)$. Also, let $\mathcal{V}=\{z\in F:v(q)\geq 0\}$, $\mathcal{V}^+=\{z\in F:v(q)>0\}$, so that $\frac{\mathcal{V}}{\mathcal{V}^+}$ is a field extension of $\mathbb{F}_p$.

\begin{lemma}\label{LI}

For all $m\in\mathbb{N}$, $z_1^{p^m},\cdots,z_r^{p^m}$ are $\mathbb{F}_p$-linearly independent modulo $\mathcal{V}^+$.

\end{lemma}

\begin{proof}

Let us suppose that there exist integers $\alpha_1,\cdots,\alpha_r\in\{0,1,\cdots,p-1\}$ such that $\alpha_1z_1^{p^m}+\cdots+\alpha_rz_r^{p^m}\in\mathcal{V}^+$, i.e. $v(\alpha_1z_1^{p^m}+\cdots+\alpha_rz_r^{p^m})>0$.\\

\noindent Firstly, note that $\alpha_i^{p^m}\equiv\alpha_i$ $($mod $p)$ by Fermat's Little theorem, so $0\equiv\alpha_1z_1^{p^m}+\cdots+\alpha_rz_r^{p^m}\equiv (\alpha_1z_1+\cdots+\alpha_rz_r)^{p^m}$ $($mod $\mathcal{V}^+)$, which implies that $\alpha_1z_1+\cdots+\alpha_rz_r\in\mathcal{V}^+$.\\

\noindent But $z_i=a^{-1}\log(\varphi(g_i)g_i^{-1})$, so $\alpha_iz_i=a^{-1}\log((\varphi(g_i)g_i^{-1})^{\alpha_i})$. But since $g\mapsto\varphi(g)g^{-1}$ defines a group homomorphism, it follows that $\alpha_iz_i=a^{-1}\log(\varphi(g_i^{\alpha_i})g_i^{-\alpha_i})$, and hence $\alpha_1z_1+\cdots+\alpha_rz_r=a^{-1}\log(\varphi(g_1^{\alpha_1})g_1^{-\alpha_1}\cdots\varphi(g_r^{\alpha_r})g_r^{-\alpha_r})=a^{-1}\log(\varphi(\underline{g}^{\alpha})g^{-\alpha})\in\mathcal{V}^+$.\\

\noindent Therefore, it follows that $v(\log(\varphi(\underline{g}^{\alpha})\underline{g}^{-\alpha}))>v(a)=\lambda$, and hence $\underline{g}^{\alpha}\in U$ by the definition of $U$. But we know that $\{g_1^p,\cdots,g_r^p,g_{r+1},\cdots,g_d\}$ is an ordered basis for $U$, which means that $p\mid\alpha_i$, i.e. $\alpha_i=0$, for all $i$ as required.\end{proof}

\begin{lemma}\label{terms}

For any $m\in\mathbb{N}$, $\alpha\in\mathbb{Z}_p^d$, $\delta^{p^m}(\underline{g}^{\alpha})=(\alpha_1'z_1^{p^m}+\cdots+\alpha_r'z_r^{p^m})\underline{g}^{\alpha}+z^{(m,\alpha)}\underline{g}^{\alpha}$ for some $z^{(m,\alpha)}\in\mathcal{V}^+$.

\end{lemma}

\begin{proof}

Firstly, we know that $\delta^{p^m}(\underline{g}^{\alpha})=(\alpha_1z_1+\cdots+\alpha_dz_d)^{p^m}\underline{g}^{\alpha}$, and we know that $v(\alpha_{r+1}z_{r+1}+\cdots+\alpha_dz_d)>0$. Also, using Lemma \ref{LI}, we see that $v(\alpha_1z_1+\cdots+\alpha_rz_r)=0$ if $p\nmid\alpha_i$ for some $i$. Therefore: 

\begin{center}
$(\alpha_1z_1+\cdots+\alpha_dz_d)^{p^m}\equiv (\alpha_1z_1+\cdots+\alpha_rz_r)^{p^m}\equiv \alpha_1^{p^m}z_1^{p^m}+\cdots+\alpha_r^{p^m}z_r^{p^m}$ $($mod $\mathcal{V}^+)$.
\end{center}

\noindent Since $\alpha_i^{p^m}\equiv\alpha_i\equiv\alpha_i'$ $($mod $p)$ for all $m$, it follows that $\alpha_1^{p^m}z_1^{p^m}+\cdots+\alpha_r^{p^m}z_r^{p^m}\equiv \alpha_1'z_1^{p^m}+\cdots+\alpha_r'z_r^{p^m}$ $($mod $\mathcal{V}^+)$.

Therefore, $(\alpha_1z_1+\cdots+\alpha_dz_d)^{p^m}=\alpha_1'z_1^{p^m}+\cdots+\alpha_r'z_r^{p^m}+z^{(m,\alpha)}$ for some $z^{(m,\alpha)}\in\mathcal{V}^+$, and hence $\delta^{p^m}(\underline{g}^{\alpha})=(\alpha_1'z_1^{p^m}+\cdots+\alpha_r'z_r^{p^m})\underline{g}^{\alpha}+z^{(m,\alpha)}\underline{g}^{\alpha}$ as required.\end{proof}

\noindent Now, define $T:=\left(\begin{array}{ccccc}z_1 & z_2 & \cdots & z_r\\ z_1^{p} & z_2^{p} & \cdots & z_r^{p}\\ . & . & \cdots & .\\ . & . & \cdots & .\\ . & . & \cdots & .\\z_1^{p^{r-1}} & z_2^{p^{r-1}} &\cdots & z_r^{p^{r-1}}\end{array}\right)$. Then using Lemma \ref{terms}, we see that:\\

\begin{equation}\label{iota}
\left(\begin{array}{c}\delta(\underline{g}^{\alpha})\\.\\.\\.\\\delta^{p^{r-1}}(\underline{g}^{\alpha})\end{array}\right)=T\left(\begin{array}{c}\alpha_1'\underline{g}^{\alpha}\\.\\.\\.\\\alpha_r'\underline{g}^{\alpha}\end{array}\right)+\left(\begin{array}{c}z^{(0,\alpha)}\underline{g}^{\alpha}\\.\\.\\.\\z^{(r-1,\alpha)}\underline{g}^{\alpha}\end{array}\right).
\end{equation}

\noindent But $T$ is a matrix of Vandermonde type, in the sense of \cite[Section 1.1]{chevalley}, and the entries of $T$ all lie in $\mathcal{V}$ and are $\mathbb{F}_p$-linearly independent modulo $\mathcal{V}^+$ by Lemma \ref{LI}, so it follows from \cite[Lemma 1.1]{chevalley} that $\det(T)$ has value zero. Therefore $T$ is invertible, and its inverse has value 0.\\

\noindent Define $u:\widehat{Q}^r\to\widehat{Q},(s_1,\cdots,s_r)^{T}\mapsto s_1+\cdots+s_r$, and define:

\begin{equation}
\overline{\iota}:\widehat{Q}\to\widehat{Q}, s\mapsto u T^{-1}\left(\begin{array}{c}\delta(\tau(s))\\.\\.\\.\\\delta^{p^{r-1}}(\tau(s))\end{array}\right)
\end{equation}

\noindent Note that $\overline{\iota}$ is continuous and $F$-linear. Also for any $\alpha\in\mathbb{Z}_p^d$, using (\ref{iota}) we see that: 

\begin{center}
$\overline{\iota}(\underline{g}^{\alpha})= u T^{-1}\left(\begin{array}{c}\delta(\underline{g}^{\alpha})\\.\\.\\.\\\delta^{p^{r-1}}(\underline{g}^{\alpha})\end{array}\right)=u\left(\begin{array}{c}\alpha_1'\underline{g}^{\alpha}\\.\\.\\.\\\alpha_r'\underline{g}^{\alpha}\end{array}\right)+uT^{-1}\left(\begin{array}{c}z^{(0,\alpha)}\underline{g}^{\alpha}\\.\\.\\.\\z^{(r-1,\alpha)}\underline{g}^{\alpha}\end{array}\right)$.
\end{center}

\noindent But the entries of $T^{-1}$ all have value at least 0, so we deduce that $\overline{\iota}(\underline{g}^{\alpha})=(\alpha_1'+\cdots+\alpha_r')\underline{g}^{\alpha}+z^{(\alpha)}\underline{g}^{\alpha}$ for some $z^{(\alpha)}\in\mathcal{V}^+$.\\

\noindent For each $m\in\mathbb{N}$, define $\iota_m:\widehat{Q}\to\widehat{Q},q\mapsto \overline{\iota}^{p^m(p-1)}(q)$, which is also continuous and $F$-linear. Also, for each $k\geq 0$, define $\widehat{Q}_k:=\{q\in Q:v(q)\geq k\}$.

\begin{proposition}\label{approximation}

Given $m\in\mathbb{N}$:

\begin{itemize}

\item The composition $\iota_m\tau:\mathcal{O}G\to\widehat{Q}$ is continuous and $\mathcal{O}$-linear.

\item There exists $k_m\in\mathbb{N}$ such that $k_m\rightarrow\infty$ as $m\rightarrow\infty$, and for all $\alpha\in\mathbb{Z}_p^d$, 

\begin{center}
$\iota_m(\underline{g}^{\alpha})\equiv\begin{cases}\underline{g}^{\alpha} & \text{if } v_p(\alpha_1+\cdots+\alpha_r)=0\\ 0 & \text{if } v_p(\alpha_1+\cdots+\alpha_r)>0\end{cases}$ $($mod $\widehat{Q}_{k_m})$.
\end{center}

\end{itemize}

\noindent Moreover, choose a sequence of integers $m_1<m_2<\cdots$ such that $k_{m_1}<k_{m_2}<\cdots$, then for every $s\in\mathcal{O}G$, $i\in\mathbb{N}$, $(\iota_{m_i}-\iota_{m_{i+1}})(\tau(s))\in\widehat{Q}_{k_{m_i}}$.

\end{proposition}

\begin{proof}

The first statement is obvious, since $\tau$ and $\iota_m$ are both continuous and $\mathcal{O}$-linear.\\

\noindent Now, we know that $\overline{\iota}$ is $F$-linear, and for any $\alpha\in\mathbb{Z}_p^d$, $\overline{\iota}(\underline{g}^{\alpha})=(\alpha_1'+\cdots+\alpha_r'+z^{(\alpha)})\underline{g}^{\alpha}$ for some $z^{(\alpha)}\in F$ with $v(z^{(\alpha)})\geq 1$. Therefore, $\iota_m(\underline{g}^{\alpha})=\overline{\iota}^{p^m(p-1)}(\underline{g}^{\alpha})=(\alpha_1'+\cdots+\alpha_r'+z^{(\alpha)})^{p^m(p-1)}\underline{g}^{\alpha}$.\\

\noindent But if $v_p(\alpha_1'+\cdots+\alpha_r')=0$ then $(\alpha_1'+\cdots+\alpha_r')^{p-1}\equiv 1$ $($mod $p)$, so $(\alpha_1'+\cdots+\alpha_r'+z^{(\alpha)})^{p-1}=1+y^{(\alpha)}$ for some $y^{(\alpha)}\in F$ with $v(y^{(\alpha)})\geq 1$, and $\iota_m(\underline{g}^{\alpha})=(1+y^{(\alpha)})^{p^m}\underline{g}^{\alpha}$.

On the other hand, if $v_p(\alpha_1'+\cdots+\alpha_r')>0$ then $v(\iota_m(\underline{g}^{\alpha}))\geq p^m(p-1)$.\\

\noindent Let  $\gamma:=\inf\{v(y^{(\alpha)}):\alpha\in\mathbb{Z}_p^d, v(\alpha_1+\cdots+\alpha_r)=0\}\geq 1$, and for each $m\in\mathbb{N}$, define $t_m:=\min\{p^m\gamma,m-v_p(k)+(p^m-k)\gamma:0<k<p^m\}$. Then since $\gamma=v(y^{(\alpha)})$ for some $\alpha\in\mathbb{Z}_p^d$, it follows from Lemma \ref{terms2} that $t_m\rightarrow\infty$ as $m\rightarrow\infty$.\\

\noindent Furthermore, also using Lemma \ref{terms2}, we see that:\\

\noindent $v((1+y^{(\alpha)})^{p^m}-1)\geq\min\{p^mv(y^{(\alpha)}),m-v_p(k)+(p^m-k)v(y^{(\alpha)}):0<k<p^m\}\geq t_m$.\\

\noindent Hence $v(\iota_m(\underline{g}^{\alpha})-\underline{g}^{\alpha})\geq t_m$ for all $m$ whenever $v(\alpha_1+\cdots+\alpha_r)=0$.\\

\noindent So, let $k_m:=\min\{p^m(p-1),t_m\}\rightarrow\infty$ as $m\rightarrow\infty$, then $v(\iota_m(\underline{g}^{\alpha}))\geq k_m$ if $v(\alpha_1+\cdots+\alpha_r)>0$, and $v(\iota_m(\underline{g}^{\alpha})-\underline{g}^{\alpha})\geq k_m$ if $v_p(\alpha_1+\cdots+\alpha_r)=0$ as required.\\

\noindent In particular, if $k_{m_1}<k_{m_2}<\cdots$, then for every $g\in G$, $\iota_{m_i}(g)\equiv\iota_{m_{i+1}}(g)$ $($mod $\widehat{Q}_{k_{m_i}})$ for all $i$, i.e. $(\iota_{m_i}-\iota_{m_{i+1}})(g)\in\widehat{Q}_{k_{m_i}}$. So since $\iota_m$ is $\mathcal{O}$-linear for every $m$, it follows that $(\iota_{m_i}-\iota_{m_{i+1}})(\tau(s))\in\widehat{Q}_{k_{m_i}}$ for every $s\in\mathcal{O}G$, and since $(\iota_{m_i}-\iota_{m_{i+1}})\tau$ is continuous, this means that $(\iota_{m_i}-\iota_{m_{i+1}})(\tau(s))\in\widehat{Q}_{k_{m_i}}$ for every $s\in\mathcal{O}G$ as required.\end{proof}

\noindent Using this proposition, it follows that there that there exists a continuous, $\mathcal{O}$-linear map $\iota:\mathcal{O}G\to\widehat{Q}$ such that $\iota(P)=0$, $\iota(s)\equiv\iota_m(\tau(s))$ $($mod $\widehat{Q}_{k_{m_i}})$ for every $i$, and for every $\alpha\in\mathbb{Z}_p^d$:

\begin{equation}
\iota(\underline{g}^{\alpha})=\begin{cases}\underline{g}^{\alpha} & \text{if } v_p(\alpha_1+\cdots+\alpha_r)=0\\ 0 & \text{if } v_p(\alpha_1+\cdots+\alpha_r)>0\end{cases}
\end{equation}

\noindent Finally, since $U$ contains $(G,G)$ and $G^p$ by Lemma \ref{faithful}, the quotient $\frac{G}{U}$ has the structure of an $\mathbb{F}_p$-vector space, with basis $\{g_1U,\cdots,g_rU\}$. Therefore, the map $\chi:\frac{G}{U}\to\frac{\mathbb{Z}}{p\mathbb{Z}},\underline{g}^{\alpha}U\mapsto\alpha_1+\cdots+\alpha_r+p\mathbb{Z}$ is a non-zero $\mathbb{F}_p$-linear map, so $\ker(\chi)=\frac{V}{U}$ for some proper open subgroup $V$ of $G$, and it follows that for all $g\in G$:

\begin{equation}\label{function}
\iota(g)=\begin{cases}g & \text{if } g\notin V\\ 0 & \text{if } g\in V\end{cases}
\end{equation}

Now we are ready to prove a control theorem. Firstly, let $C^{\infty}(G,\mathcal{O})$ be the space of locally constant functions $f:G\to\mathcal{O}$, and recall from \cite[Proposition 2.5 and Lemma 2.9]{controller} that there is a natural action $\rho:C^{\infty}(G,\mathcal{O})\to$ End$_{\mathcal{O}}\mathcal{O}G$ such that for any open subgroup $U$ of $G$, if $f\in C^{\infty}(G,\mathcal{O})$ is constant on the cosets of $U$ then the action of $f$ on $\mathcal{O}G$ can be described explicitly:\\

\noindent If $x\in\mathcal{O}G$ and $x=\underset{g\in C}{\sum}{x_gg}$, where $C$ is a set of coset representatives for $U$ in $G$ and $x_c\in\mathcal{O}U$, then

\begin{center}
$\rho(f)(x)=\underset{g\in C}{\sum}{f(g)x_cg}$
\end{center}

\noindent In particular, define $f:G\to\mathcal{O},g\mapsto\begin{cases}1 & \text{if } g\notin V \\ 0 & \text{if }g\in V\end{cases}$. Then clearly $f\in C^{\infty}(G,\mathcal{O})$ is constant on the cosets of $V$, so $\rho(f)(g)=\begin{cases}g & \text{if } g\notin V \\ 0 & \text{if }g\in V\end{cases}$.\\

\noindent Therefore, it follows from (\ref{function}) that $\iota(g)=\tau\rho(f)(g)$ for all $g\in G$, so since $\iota$, $\tau$ and $\rho(g)$ are continuous and $\mathcal{O}$-linear, it follows that $\iota=\tau\rho(f)$. Therefore, since $\iota(P)=0$, it follows that $\rho(f)(P)\subseteq P$.

\begin{proposition}

$P$ is controlled by $V$.

\end{proposition}

\begin{proof}

This is now identical to the proof of \cite[Theorem 1.4]{primitive}. Firstly, suppose that $C=\{x_1,\cdots,x_t\}$ is a complete set of coset representatives for $V$ in $G$, then for all $r\in\mathcal{O}G$, $r=\underset{i\leq t}{\sum}{r_ix_i}$ for some $r_i\in\mathcal{O}V$.\\

\noindent Suppose we can choose $C$ such that if $r\in P$ then $r_1\in P\cap\mathcal{O}V$. Then since $rx_1^{-1}x_i\in P$ for all $i=1\cdots,t$ and $rx_1^{-1}x_i$ has $x_1$ component $r_i$, it follows that $r_i\in P\cap\mathcal{O}V$ for each $i$, and hence $P$ is controlled by $V$.

It remains to prove that we can choose such a set $C$ of coset representatives such that if $\underset{i\leq t}{\sum}{r_ix_i}\in P$, then at least one of the $r_i$ lies in $P\cap\mathcal{O}V$.\\

\noindent Since $G^p\subseteq V$, it follows that $V$ has ordered basis $\{g_1^p,\cdots,g_s^p,g_{s+1},\cdots,g_d\}$ and thus $C=\{g_1^{b_1}\cdots g_s^{b_r}:0\leq b_i<p\}$ is a complete set of coset representatives for $V$ in $G$.

So for each $\underline{b}\in [p-1]^s$, let $g_{\underline{b}}=g_1^{b_1}\cdots g_r^{b_r}$ (here $[p-1]=\{0,1,\cdots,p-1\}$).\\

\noindent Then if $r=\underset{\underline{b}\in [p-1]^s}{\sum}{r_{\underline{b}}g_{\underline{b}}}\in P$, then $\rho(f)(r)=\underset{\underline{b}\in [p-1]^s}{\sum}{f(g_{\underline{b}}})r_{\underline{b}}g_{\underline{b}}$, and since $\rho(f)(P)\subseteq P$ this also lies in $P$. But $f(g_{\underline{b}})=1$ if $\underline{b}\neq 0$, and $f(g_{\underline{0}})=0$ hence $\rho(f)(r)=\underset{\underline{b}\in [p-1]^s\backslash\{\underline{0}\}}{\sum}{r_{\underline{b}}g_{\underline{b}}}\in P$.\\

\noindent Therefore, $r_{\underline{0}}g_{\underline{0}}=r-\rho(f)(r)\in P$, and thus $r_{\underline{0}}\in P\cap\mathcal{O}U$ as required.\end{proof}

\noindent So, since prime ideals in $\mathcal{O}G$ not containing $P$ correspond bijectively with prime ideals in $KG$, altogether we have now prove the following theorem:

\begin{theorem}\label{E}

Let $(G,\omega)$ be a complete, $p$-valued group of finite rank, and let $P$ be a faithful, prime ideal of $KG$. Also, let $\varphi\in$ Aut$^{\omega}(G)$ be an automorphism of $G$, and let $A$ be a closed, central subgroup of $G$ such that: 

\begin{itemize}

\item $\varphi\neq 1$.

\item $\varphi(P)=P$.

\item $\varphi(g)g^{-1}\in A$ for all $g\in G$.

\item $\varphi(a)=a$ for all $a\in A$.

\end{itemize}

\noindent Then $P$ is controlled by a proper, open subgroup of $G$.

\end{theorem}

\noindent This result is, in essence, the characteristic 0 version of \cite[Theorem B]{nilpotent}. This result was sufficient to fully prove Conjecture \ref{main} in characteristic $p$ for $G$ nilpotent in \cite{nilpotent}, but unfortunately our additional assumption that $\varphi(g)g^{-1}$ is fixed by $\varphi$ for all $g\in G$ restricts the usefulness of this result, which is why, as we will see, we cannot assume that the subgroup $A$ described in the statement of Theorem \ref{C} is central.

\subsection{Control Theorem for Prime ideals}

Now we are ready to prove Theorem \ref{C}, and the remainder of the argument is similar to the proof of \cite[Theorem 1.2]{primitive}, as given in \cite[Section 3.5]{primitive}.\\

\noindent Firstly, recall that a prime ideal $P$ of $KG$ is \emph{non-splitting} if for any closed subgroup $H$ of $G$ that controls $P$, $P\cap KH$ is a prime ideal of $KH$. Furthermore, a right ideal $I$ of $KG$ is \emph{virtually non-splitting} if $I=PKG$ for some non-splitting prime ideal $P$ of $KU$, where $U$ is some open subgroup of $G$. 

The following theorem, analogous to \cite[Theorem 5.8]{nilpotent} and partially proved in \cite[Theorem 4.8]{primitive}, essentially proves that establishing a control theorem for virtually non-splitting ideals is sufficient to establish it for all primes.

\begin{theorem}\label{split}

Let $(G,\omega)$ be a complete $p$-valued group of finite rank, let $A$ be a closed subgroup of $G$, and suppose that all faithful, virtually non-splitting right ideals of $KG$ are controlled by $A$. Then all faithful, prime ideals of $KG$ are controlled by $A$.

\end{theorem}

\begin{proof}

Let $P$ be a faithful, prime ideal of $KG$, and let $P=I_1\cap\cdots\cap I_m$ be an essential decomposition for $P$ in the sense of \cite[Definition 5.6]{McConnell}, with each $I_j$ virtually prime, and $I_1,\cdots,I_m$ forming a single $G$-orbit.\\ 

\noindent Setting $m=1$, $I_1=P$, it is clear that such a decomposition exists, so we will assume that $m$ is maximal such that a decomposition of this form exists. We know that $m$ is finite because $KG/P$ has finite uniform dimension in the sense of \cite{McConnell}. So, by \cite[Proposition 4.4]{primitive}, each $I_j$ is a virtually non-splitting right ideal of $KG$. Furthermore, since $P$ is faithful, it follows from \cite[Lemma 4.2]{primitive} that each $I_j$ is faithful.

Therefore, by assumption, $I_j$ is controlled by $A$, so $I_j=(I_j\cap KA)KG$ for each $j$. So since $P=I_1\cap\cdots\cap I_r$, we have that\\

$(P\cap KA)KG=((I_1\cap KA)\cap\cdots\cap (I_r\cap KA))KG$\\

$=(I_1\cap KA)KG\cap\cdots\cap (I_r\cap KA)KG=I_1\cap\cdots\cap I_r=P$ by \cite[Lemma 4.1($i$)]{primitive}.\\ 

\noindent Thus $P$ is controlled by $A$ as required.\end{proof}

\noindent Now, in \cite{primitive}, we defined the closed, isolated normal subgroup $C_G(Z_2(G))$ of $G$ to be 

\begin{center}
$C_G(Z_2(G)):=\{g\in G:$ if $(h,G)\subseteq Z(G)$ then $(g,h)=1\}$.
\end{center} 

\noindent The main result in that paper (\cite[Theorem 1.2]{primitive}) was that all faithful, primitive ideals in $KG$ are controlled by $C_G(Z_2(G))$. With the results proved in this section, we can now generalise this result to all prime ideals.

\begin{theorem}\label{prime-control}

Let $(G,\omega)$ be a complete $p$-valued group of finite rank, let $P$ be a faithful prime ideal of $KG$. Then $P$ is controlled by $C_G(Z_2(G))$.

\end{theorem}

\begin{proof}

First, suppose that $P$ is non-splitting, and let $H:=P^{\chi}$ be the controller subgroup of $P$:\\

\noindent Then $Q:=P\cap KH$ is a faithful, prime ideal of $KH$ by the definition of non-splitting, and since $H$ is the smallest subgroup of $G$ controlling $P$ by \cite[Theorem A]{controller}, $Q$ is not controlled by any proper subgroup of $H$. Also, note that $H$ is a normal subgroup of $G$ by the proof of \cite[Lemma 5.2]{nilpotent}, so for any $g\in G$, $(g,H)\subseteq H$.\\

\noindent Now, choose $g\in G$ such that $(g,G)\subseteq Z(G)$, and let $A:=Z(G)\cap H$. Let $\varphi$ be the automorphism of $H$ induced by conjugation by $g$. Then clearly $\varphi(Q)=Q$, and for all $h\in H$, $\varphi(h)h^{-1}=(g,h)\in Z(G)\cap H=A$. Moreover, since $A$ is central in $G$, it follows that $\varphi(a)=a$ for all $a\in A$. So applying Theorem \ref{E} gives that if $\varphi\neq 1$, then $Q$ is controlled by a proper subgroup of $H$ -- contradiction.

Therefore $\varphi=1$, i.e. $g$ centralises $H$.\\

\noindent Therefore, if $g\in G$ and $(g,G)\subseteq Z(G)$, then $(g,H)=1$, and hence $H$ is contained in $C_G(Z_2(G))$ as required. Thus $P$ is controlled by $C_G(Z_2(G))$.\\

\noindent Now suppose that $I\trianglelefteq_r KG$ is a faithful and virtually non-splitting right ideal of $KG$. Then $I=PKG$ for some open subgroup $U$ of $G$, and some faithful, non-splitting prime $P$ of $KU$.\\

\noindent We have proved that $P$ is controlled by $C_U(Z_2(U))$, and $C_U(Z_2(U))=C_G(Z_2(G))\cap U$ by \cite[Lemma 4.9]{primitive}, and hence $I$ is controlled by $C_G(Z_2(G))$.

So, using Theorem \ref{split}, it follows that every faithful, prime ideal of $KG$ is controlled by $C_G(Z_2(G))$ as required.\end{proof}

\vspace{0.2in}

\noindent Now we can finally complete the proof of Theorem \ref{C}. So suppose that $(G,\omega)$ is a nilpotent, $p$-valued group of finite rank.\\

\noindent\emph{Proof of Theorem \ref{C}.} Consider the following sequence of subgroups, $A_0=G$ and for each $i\geq 0$, $A_{i+1}=C_{A_i}(Z_2(A_i))$. Since $G$ is nilpotent, each $A_i$ must also be nilpotent.\\

\noindent Therefore, if $A_i$ is non-abelian then there must exist $g\in A_i$ with $g\notin Z(A_i)$ such that $(g,A_i)\subseteq Z(A_i)$ (i.e. $g\in Z_2(A_i)$). But by definition, if $h\in A_{i+1}=C_{A_i}(Z_2(A_i))$ then $(g,h)=1$, so since $g$ is not central, this means that $A_{i+1}\neq A_i$.\\

\noindent Moreover, if $A_i$ is abelian, then clearly $A_{i+1}=A_i$, and hence $A_j=A_i$ for all $j>i$. But since $A_{i+1}$ is a closed, isolated normal subgroup of $A_i$, the chain $G=A_0\supseteq A_1\supseteq A_2\supseteq\cdots$ must terminate, i.e. there exists $i\geq 0$ such that $A_j=A_i$ for all $j\geq i$, and hence $A_i$ is abelian. Let $A:=A_i$, and we will prove that all faithful, prime ideals in $KG$ are controlled by $A$.\\

\noindent Let $P$ be a faithful, non-splitting prime ideal of $KG$, and let us suppose for contradiction that $P$ is not controlled by $A$. But trivially, $P$ is controlled by $G=A_0$, so let $0\leq j<i$ be maximal such that $P$ is controlled by $A_j$.\\ 

\noindent Since $P$ is non-splitting and $A_j$ is a closed subgroup of $G$, $Q:=P\cap KA_j$ is a faithful prime ideal of $KA_j$, so it follows from Theorem \ref{prime-control} that $Q$ is controlled by $C_{A_j}(Z_2(A_j))=A_{j+1}$, and hence $P=QKG=(Q\cap KA_{j+1})KA_jKG=(P\cap KA_{j+1})KG$ is controlled by $A_{j+1}$ -- contradiction.

Therefore, every faithful, non-splitting prime ideal of $KG$ is controlled by $A$.\\

\noindent Now suppose that $I\trianglelefteq_r KG$ is a faithful and virtually non-splitting right ideal of $KG$. Then $I=PKG$ for some faithful, non-splitting prime $P$ of $KU$, where $U$ is some open subgroup of $G$.\\

\noindent Again, let $B_0=U$ and for $j\geq 0$ let $B_{j+1}:=C_{B_j}(Z_2(B_j))$. Then using \cite[Lemma 4.9]{primitive}, $B_j\subseteq A_j$ for each $j$, and hence $B_i$ is abelian. So since $P$ is a faithful, non-splitting prime ideal of $KU$, it follows that $P$ is controlled by $B_i\subseteq A$, and hence $I=PKG=(P\cap KB_i)KUKG\subseteq (I\cap KA)KG$ is controlled by $A$. So using Theorem \ref{split}, it follows that every faithful, prime ideal of $KG$ is controlled by $A$.\qed

\section{Dixmier Annihilators}

In this section, we will study the action of the rational Iwasawa algebra $KG$ on the affinoid Dixmier module $\widehat{D(\lambda)}$, and ultimately prove Theorem \ref{B}.\\ 

\noindent Throughout, fix $G$ a nilpotent, uniform pro-$p$ group, $\mathcal{L}=\frac{1}{p}\log(G)$, and $\mathfrak{g}=\mathcal{L}\otimes_{\mathbb{Z}_p}\mathbb{Q}_p$. We will assume further that $\mathcal{L}$ is \emph{powerful}, i.e. $[\mathcal{L},\mathcal{L}]\subseteq p\mathcal{L}$.

\subsection{Faithful Dixmier Annihilators}

Let $\lambda:\mathfrak{g}\to K$ be a $\mathbb{Q}_p$-linear map such that $\lambda(\mathcal{L})\subseteq\mathcal{O}$ and $\lambda|_{Z(\mathfrak{g})}$ is injective. Let $\mathfrak{b}$ be a polarisation of $\mathfrak{g}_K:=\mathfrak{g}\otimes_{\mathbb{Q}_p}K$ at $\lambda$, and let $\mathcal{B}:=\mathfrak{b}\cap\mathcal{L}_K$, where $\mathcal{L}_K:=\mathcal{L}\otimes_{\mathbb{Z}_p}\mathcal{O}$.\\

\noindent Fix $P:=$ Ann$_{KG}\widehat{D(\lambda)}$, which does not depend on the choice of polarisation by \cite[Theorem 4.5]{aff-dix}. 

\begin{lemma}\label{Dix-prime}

$P$ is a faithful, completely prime ideal of $KG$.

\end{lemma}

\begin{proof}

Since $\lambda|_{Z(\mathfrak{g})}$ is injective, it follows from Lemma \ref{dagger} that $P$ is a faithful ideal of $KG$.\\

Furthermore, since $\mathcal{L}_K$ is powerful, it follows from \cite[Corollary 3.4]{aff-dix} that if $I:=$ Ann$_{\widehat{U(\mathcal{L})}_K}\widehat{D(\lambda)}$ then $\frac{\widehat{U(\mathcal{L})}_K}{I}$ is a domain. So since $P=I\cap KG$, this means that $\frac{KG}{P}$ is a domain, and hence $P$ is completely prime ideal as required.\end{proof}

\noindent Therefore, using Theorem \ref{C}, it follows that $P$ is controlled by an abelian subgroup $A$ of $G$. Let $\mathcal{A}:=\frac{1}{p}\log(A)$, and let $\mathfrak{a}:=\mathcal{A}\otimes_{\mathbb{Z}_p}\mathbb{Q}_p$. Then $\mathfrak{a}$ is an abelian ideal of $\mathfrak{g}$, so since $P$ is independent of the choice of polarisation, we may assume that $\mathfrak{a}\subseteq\mathfrak{b}$. In other words, we may assume that the subgroup $A$ acts by scalars on the submodule $K_{\lambda}$ of $\widehat{D(\lambda)}=\widehat{U(\mathcal{L})}_K\otimes_{\widehat{U(\mathcal{B})}_K}K_{\lambda}$.\\

\noindent Our approach will be to study the action of $KA$ on $\widehat{D(\lambda)}$, and prove that the kernel of this action is centrally generated.\\

\noindent Let $\{x_1,\cdots,x_r\}$ be an $\mathcal{O}$-basis for $\mathcal{L}_K/\mathcal{B}$. Then using \cite[Lemma 3.3]{aff-dix}, we see that $\widehat{D(\lambda)}$ is isomorphic as a $K$-vector space to $K\langle x_1,\cdots,x_r\rangle$.

\begin{lemma}\label{action}

There exists an $\mathcal{O}$-basis $\{x_1,\cdots,x_r\}$ for $\mathcal{L}_K/\mathcal{B}$ such that if we let $\partial_i:=\frac{d}{dx_i}\in$ End$_K\widehat{D(\lambda)}$, then each $u\in\mathfrak{a}$ acts on $\widehat{D(\lambda)}$ by a polynomial $f_u\in K[\partial_1,\cdots,\partial_s]$ for some $s\leq r$ where $f_u(0)=\lambda(u)$. Moreover, if $u\in\mathcal{A}$ then $f_u\in\mathcal{O}[\partial_1,\cdots,\partial_s]$.\\

\noindent Furthermore, $s=0$ if and only if $\mathcal{A}$ is central, and for each $i=1,\cdots,s$, $\partial_i$ lies in the image of $U(\mathfrak{a})$ under the action.

\end{lemma}

\begin{proof}

Firstly, let $\mathfrak{a}^{\perp}=\{u\in\mathfrak{g}_K:\lambda([u,\mathfrak{a}])=0\}$, and let $s:=$ dim$_K\frac{\mathfrak{g}}{\mathfrak{a}^{\perp}}$. Then fix a basis $\{u_1,\cdots,u_r\}$ for $\mathcal{L}_K/\mathcal{B}$ such that $\{u_{s+1},\cdots,u_r\}$ is a basis for $(\mathfrak{a}^{\perp}\cap\mathcal{L}_K)/\mathcal{B}$. Then it follows from \cite[Proposition 3.5]{aff-dix} that each $u\in\mathfrak{a}$ acts on $\widehat{D(\lambda)}$ by a polynomial $f_u\in K[\partial_1,\cdots,\partial_s]$, and that $\partial_1,\cdots,\partial_s$ lie in the image of $U(\mathfrak{a})$ under this action.\\

\noindent Furthermore, we see using \cite[Proposition 3.3]{aff-dix} that 

\begin{center}
$f_u=\underset{\alpha\in\mathbb{N}^s}{\sum}{\frac{1}{\alpha_1!}\cdots\frac{1}{\alpha_s!}\lambda(\ad(u_s)^{\alpha_r}\cdots\ad(u_1)^{\alpha_1}(u))\partial_1^{\alpha_1}\cdots\partial_s^{\alpha_s}}$,
\end{center} 

\noindent so clearly the constant term is $\lambda(u)$. Moreover, if $u\in \mathcal{A}$ then since $\mathcal{L}$ is powerful, 

\noindent $\ad(u_r)^{\alpha_r}\cdots\ad(u_1)^{\alpha_1}(u)\in p^{\alpha_1+\cdots+\alpha_r}\mathcal{L}$, and hence $\lambda(\ad(u_r)^{\alpha_r}\cdots\ad(u_1)^{\alpha_1}(u))\in p^{\alpha_1+\cdots+\alpha_r}\mathcal{O}$ for all $\alpha\in\mathbb{N}^r$. 

So since $v_p(\alpha_i!)\leq\alpha_i$ for each $i$, it follows that $\frac{1}{\alpha_1!}\cdots\frac{1}{\alpha_r!}\lambda(\ad(u_r)^{\alpha_r}\cdots\ad(u_1)^{\alpha_1}(u))\in\mathcal{O}$ as required.\\

Finally, if $s=0$ then $\mathfrak{a}$ acts by scalars on $\widehat{D(\lambda)}$, and hence $[\mathfrak{a},\mathfrak{g}]\subseteq P$ and $\lambda([\mathfrak{a},\mathfrak{g}])=0$. So since $P$ is faithful, it follows from Lemma \ref{dagger} that $[\mathfrak{g},\mathfrak{a}]=0$, and hence $A$ is central. Conversely, if $\mathfrak{a}$ is central then clearly $\widehat{U(\mathcal{A})}_K$ acts by scalars on $\widehat{D(\lambda)}$, so since $\partial_1,\cdots,\partial_s$ lie in the image of this action, and do not act by scalars, it follows that $s=0$.\end{proof}

\noindent\textbf{Note:} It follows from this lemma that the image of $\widehat{U(\mathcal{A})}_K$ in End$_K\widehat{D(\lambda)}$ is contained in $K\langle\partial_1\cdots,\partial_s\rangle$.

\subsection{Results from rigid geometry}

We will now prove some technical results using techniques from rigid geometry. A detailed introduction to the theory of rigid geometry and its applications can be found in \cite{Bosch} and \cite{BGR}, whose results we will often use in this section.\\ 

\noindent First, recall from \cite[Definition 3.1.1]{Bosch} that an \emph{affinoid algebra} over $K$ is a quotient of the Tate algebra $K\langle t_1,\cdots, t_r\rangle$ for some $r\in\mathbb{N}$. It follows from \cite[Proposition 3.1.5]{Bosch} that any affinoid algebra $R$ carries a complete, separated filtration $w_R$.

\begin{lemma}\label{linear}

Let $\phi:K\langle u_1,\cdots,u_d\rangle\to R$ be a map of affinoid algebras, and let $a_1,\cdots,a_r\in R$ lie in the image of $\phi$. Then there exists $m\in\mathbb{N}$ such that the image of $\phi$ inside $R$ contains the affinoid $K$-subalgebra topologically generated by $\pi^m a_1,\cdots,\pi^m a_r$.

\end{lemma}

\begin{proof}

For each $i$, we know that $a_i=\phi(r_i)$ for some $r_i\in K\langle u_1,\cdots,u_d\rangle$, so choose $m\in\mathbb{N}$ such that $w_{\inf}(\pi^m r_i)\geq 0$ and $w_R(\pi^m a_i)\geq 0$ for all $i$.\\

\noindent Then there exist $K$-algebra maps $\Theta_1:K\langle X_1,\cdots,X_r\rangle\to K\langle u_1,\cdots,u_d\rangle$ and $\Theta_2:K\langle X_1,\cdots,X_r\rangle\to R$ sending $X_i$ to $\pi^m r_i$ and $\pi^m a_i$ respectively, and it is clear that $\Theta_2=\phi\Theta_1$. Therefore, the image of $\phi$ must contain the image of $\Theta_2$, which is precisely the affinoid $K$-algebra topologically generated by $\pi^m a_1,\cdots,\pi^m a_r$ as required.\end{proof}

\vspace{0.1in}

\noindent Now, let $\overline{K}$ be the algebraic closure of $K$. Recall that for each $\epsilon\in\mathbb{R}$, we define the \emph{$d$-dimensional disc of radius $\epsilon$} to be the space 

\begin{center}
$\mathbb{D}_{\epsilon}^d:=\{\underline{\alpha}\in\overline{K}^d:v_{\pi}(\alpha_i)\geq\epsilon$ for each $i\}$
\end{center}

\noindent This is an affinoid space, in the sense of \cite[Definition 3.1.1]{Bosch}, isomorphic to Sp $K\langle u_1,\cdots,u_d\rangle$. Thus all discs are isomorphic, regardless of the radius.\\

\noindent Moreover, the Tate algebra $K\langle u_1,\cdots,u_d\rangle$ can be realised as the space of analytic functions on $\mathbb{D}_{0}^d$, i.e. the set of all power series in $u_1,\cdots,u_d$ converging on the unit disc, while for each $n\in\mathbb{N}$, the subalgebra $K\langle \pi^nu_1,\cdots,\pi^nu_d\rangle$ is precisely those functions which converge on $\mathbb{D}_{-n}^d$.\\

\noindent Following \cite[5.1.2]{Lewis}, for each non-constant polynomial $g(t):=b_0+b_1t+\cdots+b_nt^n\in K[t]$ with $b_0\in\mathcal{O}$, define

\begin{center}
$\chi(g):=\underset{1\leq j\leq n}{\max}{-\frac{v_{\pi}(b_j)}{j}}$.
\end{center}

\begin{lemma}\label{scale}

Let $g(t)\in K[t]$ be a polynomial with $g(0)\in\mathcal{O}$, and let $\beta\in K$ with $v_{\pi}(\beta)> 0$. Then $\chi(\beta g)<\chi(g)$.\\ 

\noindent It follows that if $f_1(t),\cdots,f_d(t)\in K[t]$ are polynomials with $f_i(0)\in\mathcal{O}$ for each $i$, and $v(\beta)>0$ then setting $\mu_i:=\underset{1\leq j\leq d}{\max}{\chi(\beta^i f_j)}$ for each $i\geq 0$, we have that $\mu_0>\mu_1>\mu_2>\cdots$.

\end{lemma}

\begin{proof}

Suppose $g(t)=b_0+b_1t+\cdots+b_nt^n$, with $b_0\in\mathcal{O}$, $b_n\neq 0$. Then by definition; 

\begin{center}
$\chi(\beta g)=\underset{1\leq j\leq n}{\max}{-\frac{v_{\pi}(\beta b_j)}{j}}=\underset{1\leq j\leq n}{\max}{-\frac{v_{\pi}(b_j)}{j}-\frac{v_{\pi}(\beta)}{j}}$
\end{center}

So since $v_{\pi}(\beta)>0$, this maximum is strictly less than $\underset{1\leq j\leq n}{\max}{-\frac{v_{\pi}(b_j)}{j}}=\chi(g)$.\\

\noindent To prove the second statement, it suffices to prove that $\mu_0>\mu_1$ and apply induction. So suppose $\mu_1=\chi(\beta f_i)$ and $\mu_0=\chi(f_j)$, then we have that $\mu_1=\chi(\beta f_i)<\chi(f_i)\leq \chi(f_j)=\mu_0$.\end{proof}

\vspace{0.1in}

\noindent Recall from \cite[Theorem 5.1.2]{Lewis} that if we assume $b_0\neq 0$, then the set 

\begin{center}
$X(g):=\{\alpha\in\overline{K}:v_{\pi}(g(\alpha))\geq 0\}$
\end{center} 

\noindent is an affinoid subdomain of $\mathbb{A}_K^{1,an}:=\overline{K}$, whose $G$-connected component about 0 is the disc $\mathbb{D}_{\chi(g)}^1=\{\alpha\in\overline{K}:v_{\pi}(\alpha)\geq\chi(g)\}$.\\

\noindent Furthermore, if $b_0=0$, it is clear that $\chi(g)=\chi(1+g)$, and that $X(g)=\{\alpha\in\overline{K}:v_{\pi}(g(\alpha))\geq 0\}=\{\alpha\in\overline{K}:v_{\pi}(1+g(\alpha))\geq 0\}=X(1+g)$, so we reach the same conclusion.

\begin{lemma}\label{disc}

Suppose that $K$ contains a $(p-1)$'st root of $p$. Then given polynomials $f_1,\cdots,f_r\in\mathcal{O}[t]$, there exists $\alpha\in\overline{K}$, $k\in\{1,\cdots,r\}$ such that $v_p(f_i(\alpha))\geq -1$ for all $i$ and $v_p(f_k(\alpha))< \frac{1}{p-1}-1$.

\end{lemma}

\begin{proof}

If $\omega\in K$ and $\omega^{p-1}=p$ then $v_{p}(\omega)=\frac{1}{p-1}$. So for each $j\geq 0$ let

\begin{center}
$Y_j:=\{\alpha\in\overline{K}:v_{\pi}(\omega^jf_i(\alpha))\geq 0$ for all $i\}$,
\end{center} 

\noindent and set $\mu_j:=\underset{i=1,\cdots,d}{\max}{\chi(\omega^jf_i)}$. Then using \cite[Theorem 5.1.2]{Lewis} we see that $Y_j$ is an affinoid subdomain of $\mathbb{A}_1^{an}$ and the $G$-connected component of $Y_j$ about zero is the closed disc $\mathbb{D}_{\mu_j}^1$.\\

\noindent We want to find $\alpha\in\overline{K}$ such that $v_p(f_i(\alpha))\geq -1$ for all $i$, i.e. $v_p(pf_i(\alpha))\geq 0$, and since $\omega^{p-1}=p$, this just means that $\alpha\in Y_{p-1}$. So it remains to find an element $\alpha$ in the connected component $\mathbb{D}_{\mu_{p-1}}^1$ of $Y_{p-1}$ such that $v_p(f_k(\alpha))<\frac{1}{p-1}-1$ for some $k$, i.e. $v_{\pi}(f_k(\alpha))< v_{\pi}(p)(\frac{1}{p-1}-1)$.\\

\noindent For each $j\geq 0$, fix $i_j=1,\cdots,d$ such that $\chi(\omega^jf_{i_j})=\mu_j$. Using Lemma \ref{scale} we see that $\mu_0>\mu_1>\mu_2>\cdots$, and we know that for each $j$, the $G$-connected component of $X(\omega^{j-1}f_{i_{j-1}})$ about zero is $\mathbb{D}_{\mu_{j-1}}^1$. 

In particular, since $\mu_{j-1}>\mu_j$ we have that $\mathbb{D}_{\mu_{j-1}}^1\subsetneq\mathbb{D}_{\mu_j}^1$, so since $\mathbb{D}_{\mu_j}^1$ is $G$-connected, this means that $\mathbb{D}_{\mu_j}^1\not\subseteq X(\omega^{j-1}f_{i_{j-1}})$. So for each $j$, we may choose $\alpha_j\in\mathbb{D}_{\mu_j}^1\backslash X(\omega^{j-1}f_{i_{j-1}})$.\\

\noindent But $X(\omega^{j-1}f_{i_{j-1}})=\{\alpha\in\overline{K}:v_{\pi}(f_{i_{j-1}}(\alpha))\geq -(j-1)v_{\pi}(\omega)\}$, so this means that $v_{\pi}(f_{i_{j-1}}(\alpha_j))< -v_{\pi}(\omega)(j-1)$. But $v_{p}(\omega)=\frac{1}{p-1}$ so $v_{\pi}(\omega)=\frac{v_{\pi}(p)}{p-1}$, thus $v_{\pi}(f_{i_{j-1}}(\alpha_j))<-v_p(\pi)\frac{j-1}{p-1}=v_p(\pi)(\frac{1}{p-1}-\frac{j}{p-1})$\\

\noindent So, finally, choose $j=p-1$, and let $k:=i_{j-1}$. Then $\alpha_j\in\mathbb{D}_{\mu_{p-1}}^1\subseteq Y_{p-1}$ and $v_{\pi}(f_{k}(\alpha_j))<v_p(\pi)(\frac{1}{p-1}-1)$ as required.\end{proof}

\begin{corollary}\label{disc2}

Suppose that $K$ contains a $(p-1)$'st root of $p$. Then given polynomials $f_1,\cdots,f_r\in \mathcal{O}[t_1,\cdots,t_m]$, there exists $\alpha\in\overline{K}^m$, $k\in\{1,\cdots,r\}$ such that $v_p(f_i(\alpha))\geq -1$ for all $i$ and $v_p(f_k(\alpha))< \frac{1}{p-1}-1$.

\end{corollary}

\begin{proof}

If $m=1$, this is precisely Lemma \ref{disc}, so assume that $m>1$ and choose $(\alpha_1,\cdots,\alpha_{m-1})\in\mathcal{O}^{m-1}$. For each $i=1,\cdots,r$, let $g_i(t):=f_i(\alpha_1,\cdots,\alpha_{m-1},t)\in\mathcal{O}[t]$.

Then using Lemma \ref{disc}, there exists $\alpha_m\in\overline{K}$, $k\in\{1,\cdots,r\}$ such that $-1\leq v_p(g_i(\alpha_m))$ for all $i$ and $v_p(g_k(\alpha_m))\leq\frac{1}{p-1}-1$. So let $\alpha=(\alpha_1,\cdots,\alpha_m)\in\overline{K}^m$, and it follows that $-1\leq v_p(f_i(\alpha))$ for all $i$, and $v_p(f_k(\alpha))\leq\frac{1}{p-1}-1$.\end{proof}

\subsection{Almost-polynomial maps}

Now we will further explore the action of the abelian lattice $\mathcal{A}$ on the Dixmier module $\widehat{D(\lambda)}\cong K\langle x_1\cdots,x_s\rangle$.

\begin{definition}\label{almost-pol}

A map $\phi:K\langle u_1,\cdots,u_d\rangle\to K\langle t_1,\cdots,t_r\rangle$ of $K$-algebras is called an \emph{almost-polynomial map} if 

\begin{itemize}

\item $\phi(u_i)\in \mathcal{O}[t_1,\cdots,t_r]$ for each $i$,

\item $t_1,\cdots,t_r$ are contained in the image of $\phi$.

\end{itemize}

\end{definition}

\noindent Using Lemma \ref{linear}, we see that if $\phi$ is an almost-polynomial map then there exist $m\in\mathbb{N}$ such that $\im(\phi)$ contains $K\langle\pi^m t_1,\cdots,\pi^mt_r\rangle$.\\

\noindent\textbf{Example:} 1. Since $\mathcal{A}$ is abelian, we know that $\widehat{U(\mathcal{A})}_K\cong K\langle u_1,\cdots,u_d\rangle$ by \cite[Lemma 2.1]{aff-dix}, and using Lemma \ref{action} we see that the action $\rho:\widehat{U(\mathcal{A})}_K\to$ End$_K\widehat{D(\lambda)}$ has image contained in $K\langle\partial_1,\cdots,\partial_s\rangle$, and the map $K\langle u_1,\cdots,u_d\rangle\to K\langle\partial_1,\cdots,\partial_s\rangle$ is an almost-polynomial map.\\

\noindent 2. If $\phi:K\langle u_1,\cdots, u_d\rangle\to K\langle t_1,\cdots,t_r\rangle$ is an almost-polynomial map then so is the restriction $K\langle pu_1,\cdots, pu_d\rangle\to K\langle t_1,\cdots,t_r\rangle$. Moreover, if $F/K$ is a finite extension, then the scalar extension $\phi_F:F\langle u_1,\cdots,u_d\rangle\to F\langle t_1,\cdots,t_r\rangle$ is also an almost polynomial map.

\begin{lemma}\label{image}

Let $\phi:K\langle u_1,\cdots,u_d\rangle\to K\langle t_1,\cdots,t_r\rangle$ be an almost-polynomial map, and let $f_i:=\phi(u_i)\in K[t_1,\cdots,t_r]$ for each $i$. Then setting $Y:=\{\alpha\in\overline{K}^r:v_{\pi}(f_i(\alpha))\geq 0$ for each $i\}$, we have that:\\ 

$i$. $Y$ is an affinoid subdomain of $\mathbb{A}_K^{r,an}$.\\

$ii$. The image of $\phi$ is contained in the set of all functions in $K\langle t_1,\cdots,t_r\rangle$ converging 

on $Y$.

\end{lemma}

\begin{proof}

Set $A:=\im(\phi)$, then $t_1,\cdots,t_r\in A$ by Definition \ref{almost-pol}. Since $K\langle t_1,\cdots,t_r\rangle$ is affinoid, it follows from Lemma \ref{linear} that there exists $m\in\mathbb{N}$ such that $A$ contains $T=K\langle\pi^m t_1,\cdots,\pi^m t_r\rangle$, we may of course choose $m$ to be arbitrarily large.\\

\noindent If we set $B:=T\langle \zeta_1,\cdots,\zeta_d\rangle/(\zeta_i-f_i(t_1,\cdots,t_r):i=1,\cdots,d)$, then there is a natural surjection from $B$ to $A$, identical on $T$, which sends $t_i$ to $f_i(t)$. This gives rise to a closed embedding of affinoid varieties Sp $A\xhookrightarrow{}$ Sp $B$.\\

\noindent $i$. Since each $f_i$ is a polynomial, it is clear that there exists $N>0$ such that if $\alpha\in\overline{K}^r$ and $v_{\pi}(\alpha)<-N$ then $v_{\pi}(f_i(\alpha))<0$ for all $i$. So by choosing $m>N$ we may assume that

\begin{center}
$Y=\{\alpha\in\mathbb{D}_{-m}^r:v_{\pi}(f_i(\alpha))\geq 0\}$. 
\end{center}

\noindent Hence using \cite[Lemma 3.3.10($i$)]{Bosch} and the proof of \cite[Proposition 3.3.11]{Bosch}, we see that $Y=$ Sp $B$ and that $Y$ is an affinoid subdomain of $\mathbb{A}_K^{r,an}$.\\

\noindent $ii$. Notice that $K\langle u_1,\cdots,u_d\rangle$ is precisely the set of functions converging on the open unit disc $\mathbb{D}_0^d$, so it follows that the image of $K\langle u_1,\cdots,u_d\rangle$ under $\phi$ is contained in the set of functions converging on $\{\alpha\in\overline{K}^r:(f_1(\alpha),\cdots,f_d(\alpha))\in\mathbb{D}_0^d\}=\{\alpha\in\overline{K}^r:v_{\pi}(f_i(\alpha))\geq 0$ for all $i\}=Y$ as required.\end{proof}

\noindent The following result will be essential later when proving a control theorem.

\begin{proposition}\label{exp}

Let $\phi:K\langle u_1,\cdots,u_d\rangle\to K\langle t_1,\cdots,t_r\rangle$ be an almost-polynomial map, and let $f_i:=\phi(u_i)\in K[t_1,\cdots,t_r]$ for each $i$. Then there exists $k\in\{1,\cdots,d\}$ such that $\exp(pf_k(t))$ does not lie in $\phi(K\langle pu_1,\cdots,pu_d\rangle)$.

\end{proposition}

\begin{proof}

We may assume that $K$ contains a $p-1$'st root of $p$. If we prove the result in this case, then it follows generally, since if $K':=K(\sqrt[p-1]{p})$ and we can find $k$ such that $\exp(pf_k)$ does not lie in the image of $K'\langle pu_1,\cdots,pu_d\rangle$ under the scalar extension of $\phi$, then it will also not lie in the image of $K\langle pu_1,\cdots,pu_d\rangle$ under $\phi$.\\

\noindent Let $Y:=\{\alpha\in\overline{K}^r:v_{\pi}(pf_i(\alpha))\geq 0$ for all $i\}$. Then using Lemma \ref{image} we see that $Y$ is an affinoid subdomain of $\mathbb{A}_K^{1,an}$, and that $\phi(K\langle pu_1,\cdots,pu_d\rangle)$ is contained in the set of all functions in $K\langle t_1,\cdots,t_r\rangle$ converging on $Y$. So it remains to prove that for some $k$, $\exp(pf_k)$ does not converge on $Y$, and thus cannot lie in the image of $K\langle pu_1,\cdots,pu_d\rangle$.\\

\noindent Using \cite[Example 0.4.1]{exp-rad}, the disc of convergence for $\exp$ is $\{\lambda\in\overline{K}:v_p(\lambda)>\frac{1}{p-1}\}$, so it remains only to find $\alpha\in Y$ such that $v_{p}(pf_k(\alpha))\leq\frac{1}{p-1}$ for some $k$, i.e. $v_p(f_k(\alpha))\leq \frac{1}{p-1}-1$.\\

But $f_1,\cdots,f_d\in\mathcal{O}[t_1,\cdots,t_r]$, so using Corollary \ref{disc2} we know that there exists $\alpha\in\overline{K}^r$ such that $v_p(f_i(\alpha))\geq-1$ for all $i$ and $v_p(f_k(\alpha))<\frac{1}{p-1}-1$ for some $k$. Hence $v_{\pi}(pf_i(\alpha))\geq 0$ for all $i$, and hence $\alpha\in Y$, and $v_{p}(pf_k(\alpha))\leq\frac{1}{p-1}$.\end{proof}

\vspace{0.2in}

\noindent Now we will explore more closely the image of the Tate algebra under an almost polynomial map.

\begin{theorem}\label{integrally-closed}

Suppose that $\phi:K\langle u_1,\cdots,u_d\rangle\to K\langle t_1,\cdots,t_r\rangle$ is an almost-polynomial map, and let $S:=\phi(K\langle u_1,\cdots,u_d\rangle)$. Then $S$ is an integrally closed domain of Krull dimension $r$.

\end{theorem}

\begin{proof}

First, we will prove that $S$ has Krull dimension $r$, and this is very similar to the proof of \cite[Proposition 7.5]{aff-dix}:\\

\noindent Since $t_1,\cdots,t_r\in S$, it follows from Lemma \ref{linear} that $S$ contains $K\langle\pi^mt_1,\cdots,\pi^mt_r\rangle$ for some $m\geq 0$.  Therefore we have inclusions of commutative affinoid algebras, $K\langle\pi^{m}t_1,\cdots,\pi^{m}t_r\rangle\xhookrightarrow{}S\xhookrightarrow{}K\langle t_1,\cdots,t_r\rangle$, which gives rise to a chain of open embeddings of the associated affinoid spectra: Sp $K\langle t_1,\cdots,t_r\rangle\xhookrightarrow{}$ Sp $S\xhookrightarrow{}$ Sp $K\langle\pi^{m}t_1,\cdots,\pi^{m}t_r\rangle$.\\

\noindent The notion of the analytic dimension dim $X$ of a rigid variety $X$ is defined in \cite{dimension}, where it is proved to be equal to the supremum of the Krull dimensions of every affinoid algebra $R$ such that Sp $R$ is an affinoid subdomain of $X$. In particular, if Sp $B$ is an affinoid subdomain of Sp $A$ in the sense of \cite[Definition 3.3.9]{Bosch}, then K.dim$(B)\leq$ K.dim$(A)$.

Therefore, since the Tate algebras $K\langle t_1,\cdots, t_s\rangle$ and $K\langle\pi^{m}t_1,\cdots,\pi^{m}t_r\rangle$ both have dimension $r$, it remains to prove that the embeddings Sp $S\to$ Sp $K\langle\pi^{m}t_1,\cdots,\pi^{m}t_r\rangle$ and Sp $K\langle t_1,\cdots,t_r\rangle\to$ Sp $S$ define affinoid subdomains.\\

\noindent For convenience, set $D:=$ Sp $K\langle t_1,\cdots,t_r\rangle$ and $D_1:=$ Sp $K\langle\pi^{m}t_1,\cdots,\pi^{m}t_r\rangle$. Then $D$ can be realised as the unit disc in $r$-dimensional rigid $K$-space, while $D_1$ is a deformed disc containing $D$, so fixing coordinates, $D=\{(x_1,\cdots,x_r)\in D_1:v_{\pi}(x_i)\geq 0$ for all $i\}$. 

But since Sp $S$ contains $D$, we could instead write $D=\{(x_1,\cdots,x_r)\in$ Sp $S:v_{\pi}(x_i)\geq 0$ for all $i\}$, and this is a Weierstrass subdomain of Sp $S$ in the sense of \cite[Definition 3.3.7]{Bosch}, and hence $D$ is an affinoid subdomain of Sp $S$ by \cite[Proposition 3.3.11]{Bosch}. Therefore K.dim$(S)\geq$ dim$(D)=r$.\\

\noindent Now, set $T:=K\langle \pi^m t_1,\cdots,\pi^m t_r\rangle$, and let $f_i(t_1,\cdots,t_r)=\phi(u_i)\in T$ for each $i$. Define 

\begin{center}
$B:=T\langle\zeta_1,\cdots,\zeta_d\rangle/(\zeta_i-f_i(t_1,\cdots,t_r):i=1,\cdots,d)$
\end{center} 

\noindent then $B$ is an affinoid algebra which naturally surjects onto $S=\phi(K\langle u_1,\cdots,u_d\rangle)$, where each $a\in T$ is sent to $a$, and each $\zeta_i$ is sent to $\phi(u_i)$. Therefore, K.dim$(S)\leq$ K.dim$(B)$.\\

\noindent But clearly there is a map $T\to B$, inducing a map of affinoid varieties Sp $B\to$ Sp $T$, and the proof of \cite[Proposition 3.3.11]{Bosch} shows that this corresponds to the embedding of the Weierstrass subdomain $Y=\{x\in$ Sp $T: v_{\pi}(f_i(x))\geq 0$ for all $i\}$ into Sp $T$, and hence Sp $B$ is an affinoid subdomain of Sp $T=D_1$ by \cite[Proposition 3.3.11]{Bosch}, and hence K.dim$(B)\leq$ dim$(D_1)=r$.

Therefore $r\leq$ K.dim$(S)\leq$ K.dim$(B)\leq r$, forcing equality, so K.dim$(S)=r$ as required. Moreover, this implies that $S$ and $B$ have the same Krull dimension, and hence $S$ is a quotient of $B$ by a minimal prime ideal.\\

\noindent To prove that $S$ is integrally closed, we will prove that the affinoid variety $Y=$ Sp $B$ is \emph{normal}, i.e. at every point $p\in Y$, the ring of germs of affinoid functions $\mathcal{O}_{Y,p}$ (as defined in \cite[Definition 4.1]{Bosch}) is reduced and integrally closed. Using this, it will follow from \cite[Proposition 7.3.8]{BGR} that the localisation $B_{\mathfrak{q}}$ of $B$ at every prime ideal $\mathfrak{q}$ of $B$ is reduced and integrally closed. 

Since $S$ is a minimal prime quotient of $B$, it will follow that $S$ is an integrally closed domain as required.\\

\noindent Using \cite[Theorem 5.1.3]{normal}, we see that since the affine variety $\mathbb{A}_K^r$ is smooth, and hence normal, its analytification $\mathbb{A}_K^{r,an}$ is also normal. So since $Y=\{x\in\mathbb{A}_K^{r,an}: v_{\pi}(f_i(x))\geq 0$ for all $i\}$ is an affinoid subdomain of $\mathbb{A}_K^{r,an}$ by Lemma \ref{image}(ii), it follows that $\mathcal{O}_{\mathbb{A}_K^{r,an},p}=\mathcal{O}_{Y,p}$ for every point $p\in Y$. Hence $\mathcal{O}_{Y,p}$ is reduced and integrally closed as required.\end{proof}

\subsection{Using the Crossed Product}

In this subsection, we will prove a control theorem for kernels of almost-polynomial maps. Throughout, we will assume that $K$ contains a $p$'th root of unity $\zeta$.\\

\noindent Fix $A$ a free abelian pro-$p$ group of rank $d$, let $\mathcal{A}:=\frac{1}{p}\mathcal{L}_A$ be the associated $\mathbb{Z}_p$-Lie algebra of $A$, and let $\phi:\widehat{U(\mathcal{A})}_K\to K\langle t_1,\cdots,t_s\rangle$ be an almost-polynomial map.\\

\noindent Consider the crossed product $D_p=D_p(A)=\widehat{U(p\mathcal{A})}_K\ast\frac{A}{A^p}$ defined in Section \ref{dist(G)}. This is a Banach completion of $KA$ with respect to the extension of the dense embedding $\iota:KA^p\to\widehat{U(p\mathcal{A})}_K$ to $KA$, and there is a natural map $\tau:D_p\to \widehat{U(\mathcal{A})}_K$. Define $\phi':D_p\to K\langle t_1,\cdots,t_s\rangle$ and $\phi_A:KA\to K\langle t_1,\cdots,t_s\rangle$ making the following diagram commute:

\begin{center}
\begin{tikzcd}

D_p \arrow[r, "\tau"]\arrow[dr, "\phi'"] & \widehat{U(\mathcal{A})}_K\arrow[d, "\phi"]\\
 KA \arrow[u, "\iota"] \arrow[r, "\phi_A"] & K\langle t_1,\cdots,t_s\rangle

\end{tikzcd}
\end{center}

\noindent From now on, set $I=\ker(\phi')$, and let $Q:=\ker(\phi_A)=I\cap KA$, and define: 

\begin{center}
$U:=\{a\in A:\phi(a)\in\phi(\widehat{U(p\mathcal{A})}_K)\}$.
\end{center}

\begin{proposition}\label{sub-image}

$U$ is a proper open subgroup of $A$ containing $A^p$.

\end{proposition}

\begin{proof}

Since $\phi$ is a ring homomorphism, it is clear that for all $a,b\in U$, $ab\in U$, and since $KA^p$ is a subalgebra of $\widehat{U(p\mathcal{A})}_K$, it is clear that $A^p\subseteq U$. Therefore, since $\frac{A}{A^p}$ is a finite group, and $\frac{U}{A^p}$ is closed under multiplication, it follows that $U$ is a subgroup of $A$ containing $A^p$, and hence it it open.\\

\noindent Finally, since $\phi$ is an almost polynomial map, it follows from Proposition \ref{exp} that there exists $u\in\mathcal{A}$ such that $\exp(p\phi(u))=\phi(\exp(p u))$ does not lie in the image of $\widehat{U(p\mathcal{A})}_K$ under $\phi$. But $a:=\exp(pu)\in A$ and hence $a\notin U$. Therefore $U$ is a proper subgroup of $G$.\end{proof}

\noindent Using this proposition, we can fix a $\mathbb{Z}_p$-basis $\{a_1,\cdots,a_d\}$ for $A$ such that $\{a_1,\cdots,a_r,a_{r+1}^p,\cdots,a_d^p\}$ is a $\mathbb{Z}_p$-basis for $U$, so $a_1,\cdots,a_r\in U$ and $a_{r+1},\cdots,a_d\notin U$.\\

Since $A$ is a free abelian pro-$p$ group, we have that $\frac{A}{A^p}$ is a direct product of $d$ copies of the cyclic group of order $p$, where the $i$'th copy is generated by the image of $a_i$ in $\frac{A}{A^p}$. Setting $c_i:=a_iA^p$, it follows from Lemma \ref{cyclic} that:

\begin{equation}
D_p=\widehat{U(p\mathcal{A})}_K\ast\langle c_1\rangle\ast\cdots\ast\langle c_d\rangle
\end{equation}

\noindent where $\overline{c_i}^t=\overline{c_i^t}$ for $0\leq t<p$ and $\overline{c_i}^p=a_i^p$.\\

\noindent From now on, let $S:=\phi(\widehat{U(p\mathcal{A})}_K)\subseteq K\langle t_1,\cdots,t_s\rangle$, and let $B:=\widehat{U(p\mathcal{A})}_K\ast\langle c_1\rangle\ast\cdots\ast\langle c_r\rangle\leq D_p$. Then since $a_1,\cdots,a_r$ lie in $U$, the image of $B$ under $\phi$ is $S$ by the definition of $U$. Furthermore, since $KU=KA^p\ast\frac{U}{A^p}=KA^p\ast\langle c_1\rangle\ast\cdots\ast\langle c_r\rangle$, it is clear that $KU\subseteq B$.\\

\noindent Let $J:=I\cap B\trianglelefteq B$ be the kernel of the restriction of $\phi'$ to $B$, and let $I':=JD_p$ -- an ideal of $D_p$ contained in $I$.

\begin{lemma}\label{minimal}

$I$ is a prime ideal of $D_p$, minimal prime above $I'$.

\end{lemma}

\begin{proof}

Since $D_p/I\cong \im(\phi')\leq K\langle t_1,\cdots,t_s\rangle$, it is clear that $I$ is a prime ideal of $D_p$.\\

\noindent Since $D_p$ is a crossed product of $B$ with a finite group, it follows from Lemma \ref{semiprime}($ii$) that $I$ is minimal prime above $I'=(I\cap B)D_p$.\end{proof}

\noindent Now, we will need the following small result from Galois theory \cite{stack-exchange}:

\begin{lemma}\label{prime-ext}

Let $F$ be a field of characteristic 0, containing a $p$'th root of unity $\zeta$. Let $r\in F$, and suppose that $r$ has no $p$'th root in $F$. Choose a $p$'th root $\alpha\in\overline{F}$ of $r$, and let $F':=F(\alpha)$. Then if $\beta\in F'$ and $\beta^p\in F$ then $\beta=c\alpha^m$ for some $c\in F$, $0\leq m<p$.

\end{lemma}

\begin{proof}

Since $F'$ is the splitting field for the polynomial $x^p-r$ over $F$, it is clear that $F'$ is a Galois extension of $F$. So since $[F':F]=p$ this means that $Gal(F'/F)$ has order $p$.

In fact, if we consider the element $\sigma\in Gal(F'/F)$ sending $\alpha$ to $\zeta\alpha$, then $Gal(F'/F)$ is cyclic of order $p$, generated by $\sigma$.\\

\noindent The result is clear if $\beta\in F$, so assume $\beta\notin F$ and $\beta^p\in F$. Then $\beta$ is a root of the polynomial $x^p-\beta^p\in F[x]$, and hence $\sigma(\beta)$ is also a root. Therefore $\sigma(\beta)=\zeta^m \beta$ for some $0\leq m<p$, so $\sigma(\alpha^{-m}\beta)=\zeta^{-m}\alpha^{-m}\zeta^m \beta=\alpha^{-m}\beta$.\\

\noindent But since $\sigma$ generates $Gal(F'/F)$, it follows that $\alpha^{-m}\beta$ is fixed by the Galois group, so since $F'/F$ is a Galois extension, this means that $c:=\alpha^{-m}\beta\in F$, and hence $\beta=c\alpha^m$ as required.\end{proof}

\noindent For clarity, we will introduce/recall the following data:

\begin{itemize}

\item $I=\ker(\phi')\trianglelefteq D_p$.

\item $Q=I\cap KA\trianglelefteq KA$.

\item $U=\{a\in A:\phi(a)\in\phi(\widehat{U(p\mathcal{A})}_K)\}=\langle a_1,\cdots,a_r,a_{r+1}^p,\cdots,a_d^p\rangle$.

\item $B=\widehat{U(p\mathcal{A})}_K\ast\langle c_1\rangle\ast\cdots\ast\langle c_r\rangle\leq D_p$.

\item $S=\phi(\widehat{U(p\mathcal{A})}_K)=\phi'(B)\leq K\langle t_1,\cdots,t_s\rangle$.

\item $J=I\cap B\trianglelefteq B$.

\item $I'=JD_p\trianglelefteq D_p$.

\item $R:=D_p/I'$

\end{itemize}

\begin{proposition}\label{domain}

$R$ is a domain.

\end{proposition}

\begin{proof}

Since $D_p=B\ast\langle c_{r+1}\rangle\ast\cdots\ast\langle c_d\rangle$ and $I'=JD_p$, it follows from Lemma \ref{semiprime}($iii$) that $R\cong S\ast\langle c_{r+1}\rangle\ast\cdots\ast\langle c_d\rangle$, where $\bar{c}_i^p=\phi(a_i^p)$ for each $i$. So using \cite[Theorem 4.4]{Passman} we see that $R$ is reduced. Therefore, we may consider the usual semisiple artinian ring of quotients $Q(R)$ of $R$, which has the form:

\begin{center}
$Q(R)=Q(S)\ast\langle c_{r+1}\rangle\ast\cdots\ast\langle c_d\rangle$,
\end{center} 

\noindent where $Q(S)$ is the field of fractions of $S$. Note that since $S=\phi(\widehat{U(p\mathcal{H})}_K)$ is the image of a Tate algebra under an almost-polynomial map, it follows from Theorem \ref{integrally-closed} that $S$ is an integrally closed domain. It remains to prove that $Q(R)$ is a field.\\

\noindent Let $T_0:=Q(S)$, and for each $i=1,\cdots,d-r$, define $T_i:=T_{i-1}\ast\langle c_{r+i}\rangle$, so that $T_{d-r}=Q(R)$.\\

\noindent Clearly $T_0$ is a field, so we will use induction to show that $T_i$ is a field for each $i$, so in particular, $Q(R)$ is a field. So assume that for some $j>0$, $T_0,\cdots,T_{j-1}$ are all fields:\\

Then since $T_j=T_{j-1}\ast\langle c_{r+j}\rangle$ where $\bar{c}_{r+j}^p=\phi(a_{r+j}^p)\in S$, it follows that

\begin{center}
$T_j=T_{j-1}[x]/(x^p-\phi(a_{r+j}^p))$
\end{center} 

\noindent So we only need to show that the polynomial $x^p-\phi(a_{r+j}^p)\in T_{j-1}[x]$ is irreducible over the field $T_{j-1}$.\\ 

\noindent Since $K$ contains a $p$'th root of unity, we see using standard Galois theory that this just means we need to show that this polynomial has no root in $T_{j-1}$, i.e. that there is no $b\in T_{j-1}$ such that $b^p=\phi(a_{r+j}^p)$.\\

\noindent Let us suppose for contradiction that $b_1^p=\phi(a_{r+j}^p)$ for some $b_1\in T_{j-1}=T_{j-2}\ast\langle c_{r+j-1}\rangle$. Then since $\phi(a_{r+j}^p)\in S\subseteq T_{j-2}$ and $T_{j-2}$ is a field containing $K$, it follows from Lemma \ref{prime-ext} that $b_1=b_2\bar{c}_{r+j-1}^{k_1}$ for some $b_2\in T_{j-2}$, $0\leq k_1<p$.\\

\noindent Therefore, $b_2^p=\phi((a_{r+j}a_{r+j-1}^{-k_1})^p)\in S$, so applying a second induction, for each $i>0$, we can find integers $0\leq k_1,\cdots,k_{i-1}<p$ and $b_i\in T_{j-i}$ such that $b_i^p=\phi((a_{r+j}a_{r+j-1}^{-k_1}a_{r+j-2}^{-k_2}\cdots a_{r+j-i+1}^{-k_{i-1}})^p)\in S$.

Taking $i=j$ we have that $b_j\in T_0=Q(S)$ and $b_j^p\in S$. So since $S$ is integrally closed, it follows that $b_j\in S\subseteq K\langle t_1,\cdots,t_s\rangle$.\\

\noindent Now, $(b_j\phi(a_{r+j}^{-1}a_{r+j-1}^{k_1}\cdots a_{r+1}^{k_{j-1}}))^p=1$, so it follows that there is a $p$'th root of unity $\zeta\in K$ such that: 

\begin{center}
$\zeta b_j=\phi(a_{r+j}a_{r+j-1}^{-k_1}a_{r+j-1}^{-k_2}\cdots a_{r+1}^{-k_{j-1}})$.
\end{center}

\noindent Therefore, since $b_j\in S$, this means that $\phi(a_{r+j}a_{r+j-1}^{-k_1}a_{r+j-1}^{-k_2}\cdots a_{r+1}^{-k_{j-1}})\in S=\phi(\widehat{U(p\mathcal{A})}_K)$, or in other words $a_{r+j}a_{r+j-1}^{-k_1}a_{r+j-1}^{-k_2}\cdots a_{r+1}^{-k_{j-1}}\in U$ by the definition of $U$.\\

This is the required contradiction since $\{a_1,\cdots,a_r,a_{r+1}^p,\cdots,a_d^p\}$ is a $\mathbb{Z}_p$-basis for $U$, and each $k_i$ is less than $p$.\end{proof}

\begin{theorem}\label{D}

Let $\phi:\widehat{U(\mathcal{A})}_K\to K\langle t_1,\cdots,t_s\rangle$ be an almost-polynomial map. Then the kernel $Q$ of the restriction of this map to $KA$ is controlled by $U$.

\end{theorem}

\begin{proof}

If $I=\ker(\phi')\trianglelefteq D_p$, then using Proposition \ref{domain} we see that $R=D_p/(I\cap B)D_p$ is a domain. But we know that $I$ is minimal prime above $(I\cap B)D_p$ by Lemma \ref{minimal}, so it follows that $I=(I\cap B)D_p$.\\

\noindent So, if $r\in Q=I\cap KA$ then since $KA=KU\ast\frac{A}{U}$, $r=\underset{a\in A//U}{\sum}{s_a a}$ for some $s_{a}\in KU\subseteq B$. So since $r\in I=(I\cap B)D_p$ it follows that $s_a\in I\cap B\cap KU=Q\cap KU$ for each $a$, and hence $r\in (Q\cap KU)KA$. Since our choice of $r$ was arbitrary, this means that $Q=(Q\cap KU)KA$, i.e. $Q$ is controlled by $U$.\end{proof}

\subsection{Control Theorem for Dixmier annihilators}

Now we can finally conclude our proof of Theorem \ref{B}. Again, $G$ is a nilpotent, uniform pro-$p$ group, whose $\mathbb{Z}_p$-Lie algebra $\mathcal{L}=\frac{1}{p}\log(G)$ is powerful.\\ 

\noindent Fix a linear form $\lambda:\mathcal{L}\to \mathcal{O}$ such that the restriction of $\lambda$ to $Z(\mathfrak{g})$ is injective, then $P:=$ Ann$_{KG}\widehat{D(\lambda)}$ is a faithful prime ideal of $KG$ by Lemma \ref{Dix-prime}, and $\frac{KG}{P}$ is a domain.

\vspace{0.2in}

\noindent\emph{Proof of Theorem \ref{B}} Firstly, for any finite extension $F/K$, if we let $I=$ Ann$_{FG}\widehat{D(\lambda_F)}$, where $\lambda_F$ is the scalar extension of $\lambda$, then clearly $I\cap KG=P$. So if we prove that $I$ is controlled by $Z(G)$, then it will follow from Lemma \ref{base-change} that $P$ is controlled by $Z(G)$. Therefore, we may pass to field extensions of $K$ without issue. In particular we may assume that $F=K$ contains a $p$'th root of unity.\\

\noindent Since $P$ is a faithful, prime ideal of $KG$, it follows from Theorem \ref{C} that $P$ is controlled by an abelian subgroup of $G$. So let $A=P^{\chi}$ be the controller subgroup of $P$, then $A$ is an abelian normal subgroup of $G$, so if we let $\mathcal{A}:=\frac{1}{p}\log(A)$ then $\mathcal{A}$ is an abelian ideal of $\mathcal{L}$. We want to prove that $A$ is central in $G$, or equivalently that $\mathcal{A}$ is central in $\mathcal{L}$.\\

\noindent Using Lemma \ref{action}, we see that the image of $\widehat{U(\mathcal{A})}_K$ in End$_K\widehat{D(\lambda)}$ is contained in a Tate algebra $K\langle\partial_1,\cdots,\partial_s\rangle$ such that if $s>0$ then the map $\widehat{U(\mathcal{A})}_K\to K\langle\partial_1,\cdots,\partial_s\rangle$ is an almost polynomial map. Moreover, $s=0$ if and only if $\mathcal{A}$ is central, so let us assume that $s>0$.\\

\noindent Then using Proposition \ref{sub-image} and Theorem \ref{D}, we can find a proper, open subgroup $U$ of $G$ such that the kernel $Q$ of the restriction of $\phi$ to $KA$ is controlled by $U$. But clearly $Q=P\cap KA$, so this is a contradiction since $A$ is the controller subgroup of $P$.\qed

\section{Primitive Ideals}

The aim of this section is to prove our main result Theorem \ref{A}. The essence of our argument is to compare general primitive ideals in $KG$ to Dixmier annihilators.

\subsection{Weakly Rational ideals}

Fix $G$ a uniform, nilpotent pro-$p$ group, and as usual let $\mathcal{L}:=\frac{1}{p}\mathcal{L}_G$, and $\mathfrak{g}:=\mathcal{L}\otimes_{\mathbb{Z}_p}\mathbb{Q}_p$.

\begin{definition}

Given a prime ideal $P$ of $KG$, we say that $P$ is \emph{weakly rational} if $Z(KG/P)$ is a finite field extension of $K$.

\end{definition}

It follows from \cite[Theorem 1.1(1)]{endomorphism} that any primitive ideal of $KG$ is weakly rational.

\begin{lemma}\label{faithful-max}

Let $P\subseteq Q$ be weakly rational ideals of $KG$, and suppose that $P$ is faithful. Then $Q$ is faithful.

\end{lemma}

\begin{proof}

Let $F_1=Z(KG/P)$, $F_2=Z(KG/Q)$, then $F_1,F_2$ are finite field extensions of $K$, and clearly the natural surjection $KG/P\twoheadrightarrow KG/Q$ reduces to a field extension $F_1\xhookrightarrow{} F_2$.\\ 

\noindent Since $Q^{\dagger}=\{g\in G:g-1\in Q\}$ is a normal subgroup of $G$, using nilpotence of $G$ we see that if $Q^{\dagger}\neq 1$, then there exists $z\in Q^{\dagger}\cap Z(G)$ with $z\neq 1$. Thus $z+P,1+P\in F_1\subseteq F_2$, and $z+Q=1+Q$, which implies that $z+P=1+P$ and hence $z-1\in P$. So since $P$ is faithful, $z=1$ -- contradiction.

Therefore, $Q^{\dagger}=1$, and hence $Q$ is faithful.\end{proof}

\noindent Now, recall from Section \ref{dist(G)} the definition of the Banach completions $D_{p^n}=\widehat{U(p^n\mathcal{L})}_K\ast\frac{G}{G^{p^n}}$ of $KG$ for each $n\in\mathbb{N}$.

\begin{proposition}\label{prim-max}

Let $P$ be a primitive ideal of $KG$, then for all sufficiently high $n\in\mathbb{N}$, there exists a primitive ideal $Q_n$ of $D_{p^n}=\widehat{U(p^n\mathcal{L})}_K\ast\frac{G}{G^{p^n}}$ such that $Q_n\cap KG=P$.

\end{proposition}

\begin{proof}

Since $P$ is primitive, $P=$ Ann$_{KG}M$ for some irreducible $KG$-module $M$. Using \cite[Proposition 10.6(e), Corollary 10.11]{annals}, we see that for $n$ sufficiently high, $\widehat{M}:=D_{p^n}\otimes_{KG}M\neq 0$.\\

\noindent Since $M$ is irreducible and $\widehat{M}\neq 0$, the natural map $M\to\widehat{M},m\mapsto 1\otimes m$ is injective. And since $D_{p^n}$ is a Banach completion of $KG$ with respect to some filtration $w$, it follows that $\widehat{M}$ is a completion of $M=KGm$ with respect to the filtration $v(rm)=\sup\{w_n(r+y):y\in KG$ and $ym=0\}$.

Therefore, if $r\in P$, i.e. $rM=0$, then taking limits shows that $r\widehat{M}=0$, so $P\subseteq$ Ann$_{KG}\widehat{M}=($Ann$_{D_{p^n}}\widehat{M})\cap KG$.\\

\noindent Now, since $D_{p^n}$ is Noetherian and $\widehat{M}$ is a finitely generated $D_{p^n}$-module, we can choose a maximal submodule $U\leq\widehat{M}$, and let $M':=\widehat{M}/U$ -- an irreducible $D_{p^n}$-module.\\

\noindent Since $M$ is irreducible, the composition $M\xhookrightarrow{}\widehat{M}\twoheadrightarrow M'$ is either injective or zero. If it is zero then $M\subseteq U$, and hence $\widehat{M}\subseteq U$ and $M'=0$. This contradiction implies that the composition is injective.\\

Finally, let $Q_n=$ Ann$_{D_{p^n}}M'$, then $Q_n$ is a primitive ideal of $D_{p^n}$, and $P\subseteq$ Ann$_{KG}\widehat{M}\subseteq$ Ann$_{KG}M'=Q_n\cap KG$. Also, if $r\in Q_n\cap KG$ then $rM'=0$, so since $M\subseteq M'$, $rM=0$ and $r\in P$. Thus $P=Q_n\cap KG$ as required.\end{proof}

\subsection{Reduction from $KG$ to $KG^{p^n}$}

Now we begin to explore how we can relate primitive ideals in Iwasawa algebras to Dixmier annihilators:

\begin{proposition}\label{Dix-decomp}

Given a primitive ideal $P$ of $KG$, there exists $m\in\mathbb{N}$ with $m\geq 1$, finite extensions $F_{1},\cdots,F_{r}/K$ and $\mathbb{Q}_p$-linear maps $\lambda_{i}:\mathfrak{g}\to F_{i}$ with $\lambda_{i}(p^m\mathcal{L})\subseteq\mathcal{O}_{F_i}$ for each $i=1,\cdots,r$, such that:

\begin{center}
$P\cap KG^{p^m}=$\emph{ Ann}$_{KG^{p^m}}\widehat{D(\lambda_1)}_{F_{1}}\cap\cdots\cap$\emph{ Ann}$_{KG^{p^m}}\widehat{D(\lambda_{r})}_{F_{r}}$
\end{center}

\end{proposition}

\begin{proof}

Using Proposition \ref{prim-max}, if $P$ is primitive, then for any sufficiently high $n\geq 1$, there is a primitive ideal $Q$ of $D_{p^n}=\widehat{U(p^n\mathcal{L})}_K\ast\frac{G}{G^{p^n}}$ such that $Q\cap KG=P$, and hence $Q\cap KG^{p^n}=P\cap KG^{p^n}$.\\

\noindent Let $I=Q\cap\widehat{U(p^n\mathcal{L})}_K$, then using Lemma \ref{semiprimitive} we see that $I$ is a semiprimitive ideal of $\widehat{U(p^n\mathcal{L})}_K$, so choose primitive ideals $J_{1},J_{2},\cdots,J_{r}$ of $\widehat{U(p^n\mathcal{L})}_K$ such that $I=J_{1}\cap J_{2}\cap\cdots\cap J_{r}$.\\

\noindent Since each $J_{i}$ is primitive, it follows from \cite[Theorem A]{aff-dix} that there exists $m\geq n$ such that for each $i$, $J_{i}\cap\widehat{U(p^{m}\mathcal{L})}_K=$ Ann$_{\widehat{U(p^m\mathcal{L})}_K}\widehat{D(\lambda_i)}_{F_i}$ for $F_i/K$ a finite extension, $\lambda_i:\mathfrak{g}\to F_i$ $\mathbb{Q}_p$-linear with $\lambda_i(p^m\mathcal{L})\subseteq\mathcal{O}_{F_i}$. Thus:\\

\noindent $P\cap KG^{p^{m}}=Q\cap KG^{p^{m}}=I\cap KG^{p^{m}}=(J_{1}\cap\widehat{U(p^{m}\mathcal{L})}_K)\cap\cdots\cap (J_{r}\cap\widehat{U(p^{m}\mathcal{L})}_K)\cap KG^{p^{m}}$ is an intersection of Dixmier annihilators as required.\end{proof}

\noindent Now, we want to show that all faithful, primitive ideals of $KG$ are centrally generated, which we know is true for Dixmier annihilators by Theorem \ref{B}. Proposition \ref{Dix-decomp} allows us to compare general primitive ideals to Dixmier annihilators, and the following result uses this to prove a reduced version of Theorem \ref{A}:

\begin{theorem}\label{reduced-version}

Let $G$ be a nilpotent, uniform pro-$p$ group with centre $Z$, and let $P$ be a faithful, primitive ideal of $KG$. Then there exists $N\in\mathbb{N}$ such that for all $n\geq N$, $P\cap KG^{p^n}$ is controlled by $Z^{p^n}$.

\end{theorem}

\begin{proof}

Using Proposition \ref{Dix-decomp}, we see that for some $m\geq 1$, there are finite extensions $F_1,\cdots,F_r$ and $\mathbb{Q}_p$-linear maps $\lambda_{i}:\mathfrak{g}\to F_{i}$ with $\lambda(p^m\mathcal{L})\subseteq\mathcal{O}_{F_i}$ such that $P\cap KG^{p^m}=$ Ann$_{KG^{p^m}}\widehat{D(\lambda_{1})}_{F_{1}}\cap\cdots\cap$ Ann$_{KG^{p^m}}\widehat{D(\lambda_{r})}_{F_{r}}$.\\

\noindent For each $i=1,\cdots, r$, set $J_{i}:=$ Ann$_{KG^{p^m}}\widehat{D(\lambda_{i})}_{F_{i}}$ for convenience, clearly these are prime ideals of $KG^{p^m}$, thus $P\cap KG^{p^m}$ is semiprime and $J_{1},\cdots,J_{r}$ are the minimal primes above $P\cap KG^{p^m}$, hence they are all $G$-conjugate by the proof of \cite[Lemma 5.4(b)]{nilpotent}. Note that for all $n\geq m$, $J_i\cap KG^{p^n}=$ Ann$_{KG^{p^n}}\widehat{D(\lambda)}_{F_i}$ for each $i$.\\

\noindent Also, since $P$ is faithful, $P\cap KG^{p^m}$ is faithful, so $J_{1}^{\dagger}\cap\cdots\cap J_{r}^{\dagger}=P^{\dagger}=1$. But since $J_{1}^{\dagger},\cdots,J_{r}^{\dagger}$ are $G$-conjugate and $G$ is orbitally sound by \cite[Proposition 5.9]{nilpotent}, this means that the subgroup $1$ must have finite index in $J_{i}^{\dagger}$ for each $i$, which means that they are finite. But $G$ is torsionfree, thus $J_{i}^{\dagger}=1$ for all $i$, i.e. $J_{1},\cdots,J_{r}$ are faithful.

So since $J_{i}=$ Ann$_{KG^{p^m}}\widehat{D(\lambda_{i})}_{F_{i}}$ is faithful, it follows from Lemma \ref{dagger} that $\lambda_{i}$ is injective when restricted to $Z(\mathfrak{g})$.\\ 

\noindent Now, since $m\geq 1$, note that for all $n\geq m$, $\frac{1}{p}\log(G^{p^n})=p^{n-1}\log(G)$ is a powerful Lie lattice. Therefore, using Theorem \ref{B}, we see that $J_{i}\cap KG^{p^{n}}=$ Ann$_{KG^{p^n}}\widehat{D(\lambda_i)}_{F_i}$ is controlled by $Z(G^{p^n})$ for each $i$, and using \cite[Lemma 8.4(a)]{nilpotent}, $Z(G^{p^n})=Z(G)\cap G^{p^n}=Z^{p^n}$.\\

\noindent Therefore, setting $B_{i,n}:=J_{i}\cap KG^{p^n}=$ Ann$_{KG^{p^n}}\widehat{D(\lambda_i)}$, $B_{i,n}=(B_{i,n}\cap KZ^{p^n})KG^{p^n}$ for each $i$, so using \cite[Lemma 4.1(a)]{primitive}:\\

$P\cap KG^{p^n}=B_{1,n}\cap\cdots\cap B_{r,n}=(B_{1,n}\cap KZ^{p^n})KG^{p^n}\cap\cdots\cap (B_{r,n}\cap KZ^{p^n})KG^{p^n}$\\

$=(B_{1,n}\cap\cdots\cap B_{r,n}\cap KZ^{p^n})KG^{p^n}=(P\cap KZ^{p^n})KG^{p^n}$\\

\noindent Hence $P\cap KG^{p^n}$ is controlled by $Z^{p^n}$ as required.\end{proof}

\subsection{Extension from $KG^{p^n}$ to $KG$}

The results of the previous subsection show that we can establish Theorem \ref{A} after passing to $G^{p^n}$ for sufficiently high $n$. We now just need to extend to $KG$.

\begin{lemma}\label{rational}

Let $P$ be a weakly rational ideal of $KG$. Then $P\cap KZ(G)$ is a maximal ideal of $KZ(G)$.

\end{lemma}

\begin{proof}

Since $P$ is prime in $KG$, $Q:=P\cap KZ(G)$ is prime in $KZ(G)$. So setting $F:=Z(KG/P)$, it is clear that $KZ(G)/Q\xhookrightarrow{} F$. So since $KZ(G)/Q$ is a domain containing $K$, and $F$ is a finite extension of $K$, it follows that $KZ(G)/Q$ is a field, and hence $Q$ is maximal.\end{proof}

\begin{proposition}\label{equal-int}

Let $G$ be a nilpotent, uniform pro-$p$ group, and let $P_1\subseteq P_2$ be faithful, primitive ideals of $KG$. Then there exists $n\in\mathbb{N}$ such that $P_1\cap KG^{p^n}=P_2\cap KG^{p^n}$. It follows that if $P$ is a faithful, primitive ideal of $KG$ then $P$ is maximal.

\end{proposition}

\begin{proof}

Using Theorem \ref{reduced-version}, we see that there exist $N_1,N_2\in\mathbb{N}$ such that for all $n_1\geq N_1$, $n_2\geq N_2$, $P_i\cap KG^{p^{n_i}}$ is controlled by $Z(G)^{p^{n_i}}$ for each $i$. So choose $n\geq\max\{N_1,N_2\}$ and we have that $P_1\cap KG^{p^n},P_2\cap KG^{p^n}$ are controlled by $Z(G)^{p^n}$.\\

\noindent Since $P_1$ is primitive, it is weakly rational, so using Lemma \ref{rational} we see that $P_1\cap KZ(G)$ is a maximal ideal of $KZ(G)$. So since $P_1\cap KZ(G)\subseteq P_2\cap KZ(G)$, we have that $P_1\cap KZ(G)=P_2\cap KZ(G)$, and hence $P_1\cap KZ(G)^{p^n}=P_2\cap KZ(G)^{p^n}$. Therefore:

\begin{center}
$P_1\cap KG^{p^n}=(P_1\cap KZ(G)^{p^n})KG^{p^n}=(P_2\cap KZ(G)^{p^n})KG^{p^n}=P_2\cap KG^{p^n}$.
\end{center}

\noindent Finally, given a faithful, primitive ideal $P$ of $KG$, let $Q$ be a maximal ideal of $KG$ containing $P$. Since $P$ and $Q$ are primitive, they are weakly rational, so since $P$ is faithful, $Q$ is faithful by Lemma \ref{faithful-max}. Thus, by the above, there exists $n\in\mathbb{N}$ such that $P\cap KG^{p^n}=Q\cap KG^{p^n}$ is controlled by $Z(G)^{p^n}$.\\

\noindent But $P\cap KZ(G)$ is prime in $KZ(G)$, so $P\cap KG^{p^n}=(P\cap KZ(G)^{p^n})KG^{p^n}$ is prime in $KG^{p^n}$ by Theorem \ref{completely-prime'}. So since $P\cap KG^{p^n}=Q\cap KG^{p^n}$, it follows from \cite[Theorem 16.6($iii$)]{Passman} that $P=Q$, and hence $P$ is maximal.\end{proof}

\begin{theorem}\label{d-standard}

Let $G$ be a nilpotent, uniform pro-$p$ group. Then all faithful, primitive ideals of $KG$ are controlled by $Z(G)$.

\end{theorem}

\begin{proof}

Let $P$ be a faithful, primitive ideal of $KG$, and let $Z=Z(G)$. We want to prove that $P$ is controlled by $Z$.\\

\noindent Using Theorem \ref{reduced-version}, we know that there exists $n\in\mathbb{N}$ such that $P\cap KG^{p^n}$ is controlled by $Z^{p^n}$, and hence is prime in $KG^{p^n}$ by Theorem \ref{completely-prime'}. So let $J:=(P\cap KG^{p^n})KG$, then using Lemma \ref{semiprime} we see that $J$ is a semiprime ideal of $KG$, and $P$ is minimal prime above $J$.\\

\noindent Let $Q:=P\cap KZ$, then $Q$ is prime in $KZ$, so $QKG$ is prime in $KG$ by Theorem \ref{completely-prime'}. And since $P\cap KG^{p^n}=(P\cap KZ^{p^n})KG^{p^n}$, we have that:

\begin{center}
$J=(P\cap KG^{p^n})KG=(P\cap KZ^{p^n})KG\subseteq QKG$.
\end{center}

\noindent But clearly $QKG\subseteq P$, so since $QKG$ is prime and $P$ is minimal prime above $J$, it follows that $P=QKG=(P\cap KZ)KG$, and hence $P$ is controlled by $Z$ as required.\end{proof}

\noindent Now we can finally prove our main result. First, we just need a small Lemma:

\begin{lemma}\label{dagger2}

Let $G$ be a uniform pro-$p$ group, let $N$ be a closed, normal subgroup of $G$. Then there exists an open, uniform normal subgroup $U$ of $G$ such that $N\cap U$ is a closed, isolated normal subgroup of $U$.

\end{lemma}

\begin{proof}

Recall from \cite[Definition 1.6]{Billy} the definition of the \emph{isolater} $i_G(N)$ of $N$ in $G$, and recall from \cite[Proposition 1.7, Lemma 1.8]{Billy} that $i_G(N)$ is a closed, isolated normal subgroup of $G$, and $N$ is open in $i_G(N)$.\\

\noindent Therefore, there exists $n\in\mathbb{N}$ such that if $g\in i_G(N)$ then $g^{p^n}\in N$. So if $g=h^{p^n}\in U:=G^{p^n}$ and $g^p=h^{p^{n+1}}\in N\subseteq i_G(N)$, then $h\in i_G(N)$, so $g=h^{p^n}\in N$. Hence $N\cap U$ is isolated in $U$ as required.\end{proof}

\vspace{0.2in}

\noindent\emph{Proof of Theorem \ref{A}.} Let $P$ be a primitive ideal of $KG$, and we want to prove that $P$ is virtually standard, i.e. that $P\cap KU$ is a finite intersection of standard ideals for some open, normal subgroup $U$ of $G$.\\

\noindent Firstly, if $P$ is faithful, then it follows from Theorem \ref{d-standard} that $P$ is controlled by $Z(G)$, and hence is standard, and using Proposition \ref{equal-int} we see that $P$ is maximal. So we can assume that $P$ is not faithful.\\

\noindent Let $N:=P^{\dagger}=\{g\in G:g-1\in P\}$. Then $N$ is a closed, normal subgroup of $G$, so by Lemma \ref{dagger2}, there exists an open, uniform normal subgroup $U$ of $G$ such that $N\cap U$ is isolated in $U$. Let $Q:=P\cap KU$, then $Q$ is a semiprimitive ideal in $KU$ by Lemma \ref{semiprimitive}, and $Q^{\dagger}=N\cap U$ is a closed, isolated normal subgroup of $U$.\\

\noindent Let $U_1:=\frac{U}{Q^{\dagger}}$, and let $Q_1:=\frac{Q}{(Q^{\dagger}-1)KU}$. Then $U_1$ is a nilpotent, uniform pro-$p$ group and $Q_1$ is a faithful semiprimitive ideal of $KU_1$. Therefore, it follows that $Q$ is a finite intersection of faithful, primitive ideals in $KU_1$. Since all faithful primitives in $KU_1$ are maximal and standard, this means that $Q_1$ is a finite intersection of maximal standard ideals.\\

\noindent Therefore, since $Q_1$ is a homomorphic image of $Q$, this means that $Q$ is a finite intersection of maximal, standard ideals, and it follows from Definition \ref{standard} that $P$ is a virtually standard prime ideal of $KG$. Therefore, it remains to show that $P$ is maximal.\\

\noindent Using Lemma \ref{semiprime}($ii$), we see that $P$ is minimal prime above the semiprime ideal $(P\cap KU)KG$. So since $P\cap KU$ is semimaximal in $KU$, it follows from \cite[Theorem 16.6($iii$)]{Passman} that $P$ is maximal in $KG$ as required.\qed

\bibliographystyle{abbrv}

\end{document}